\documentclass{article}
\usepackage{amssymb,amsmath,amsthm,graphicx, color, amsfonts, color, authblk, relsize}
\usepackage[english]{babel}
\usepackage{geometry}

\def\supp{\mathop{\rm supp}\nolimits}
\newtheorem{theorem}{Theorem}[section]

\newtheorem{lemma}[theorem]{Lemma}
\newtheorem{proposition}[theorem]{Proposition}
\newtheorem{corollary}[theorem]{Corollary}
\newtheorem{definition}[theorem]{Definition}
\newtheorem{remark}[theorem]{Remark}
\newtheorem{example}[theorem]{Example}

\renewenvironment{proof}[1][.]{%
\bigskip\noindent{\bf Proof#1 }}{%
\hfill$\blacksquare$\bigskip}

\begin{document}

\title{On the connection between a skew product IFS and the ergodic optimization for a finite family of  potentials}

\author[1]{Elismar R. Oliveira\thanks{Email: elismar.oliveira@ufrgs.br}
}

\affil[1]{Universidade Federal do Rio Grande do Sul\\
Instituto de Matem\'atica e Estat\'istica - UFRGS\\
Av. Bento Gon\c calves, 9500,
Porto Alegre - 91500 - RS -Brasil}

\date{\today}
\maketitle

\begin{abstract}
  We study a skew product IFS on the cylinder defined by Baker-like maps associated to a finite family of potential functions and the doubling map. We show that there exist a compact invariant set with attractive behavior and a random SRB measure whose support is in that set. We also study the IFS ergodic optimization problem for that finite family of potential functions and characterize the maximizing measures and the critical value through a discounting limit. This shows the connection between this maximization problem and the superior boundary of the compact invariant set, which is described as a graph  of the solution  of the Bellman equation.
\end{abstract}

\footnotetext{\textup{2010} \textit{Mathematics Subject Classification}: 37-XX, 28Dxx, 37B10, 37Hxx, 37B55, 49Lxx, 49L20, 90C39, 37L40.}

\section{Introduction}

This paper  studies  the  dynamics of a skew IFS
$G_{c}(x,y)= (T(x) , A_{c}(x)+\lambda y), c \in \mathcal{C}$ for a finite
family $\mathcal{C}$ of $m \geq 2$ potential functions $A_{c}:\textbf{S}^{1}
\to \mathbb{R}$, where $T: \textbf{S}^{1} \to \textbf{S}^{1}$ is the doubling
map on the circle $\textbf{S}^{1}$. The study is made
from  measure theoretical point
of view. The remarkable fact is that this problem has a structure that is
very close to the Lagrangian dynamic problem, where the circle $S^1$ plays
the role of a manifold and the symbolic space plays the role of the velocity
and defines a vertical coordinate on the cylinder. Thanks to that analogy, we are able to employ the same
variational methods used to study the Aubry-Mather theory, such as viscosity
solutions and discounting limits, to characterize the associated
optimization.

On one hand we have a geometrical property which is the existence
of an invariant subset $\Lambda$ of the cylinder having a random SRB
measure, which we are going to define later on. On the other
hand, we establish a subtle and non-trivial relation between the
boundary of this set and the analogous of the Aubry-Mather theory. More
specifically, as in the Lagrangian case, the solution of the Bellman equation
with a discounting parameter $\lambda \to 1$ plays the role of the viscosity
solution of the Hamilton-Jacobi equation with a viscosity parameter
$\alpha=1-\lambda \to 0$, allowing us to employ the Fenchel-Rockafellar
duality theorem to relate the Bellman equation solution and the problem of
maximizing the integral of a ``Lagrangian function" which in our case is a
parametric family of potential functions. The key element here is the set of
measures over which the maximization is performed. Thanks to the analogy with the
Lagrangian case, we are able to define the concept of holonomic
measure in our setting, which turns out to be the right choice
to reproduce the expected results such as the characterization of the
critical value and the support of the maximizing measures.

\subsection{References and main questions}

In \cite{MR2201152} the authors study,  from a topological point of view, the
dynamics of the skew product $G: X\times \mathbb{R} \to X\times \mathbb{R}$
given by $G(x,y)= (T(x) , A(x)+\lambda y)$, where $0< \lambda <1$, $X=
\mathbb{R}/\mathbb{Z}=\textbf{S}^1$, $T$ is the multiplication by some $l
\geq 2$ and $A:X \to \mathbb{R}$ is a Lipschitz potential function. These
maps are also called linear Baker maps. The goal of that paper was to prove
that there exists a topological attractor  which is
homeomorphic to an annulus when the discounting limit is taken, that is, for
$\lambda \to 1$. In \cite{MR1862809} this map was studied from a
measure theoretical point of view, giving a
description of the SRB measure. This kind of skew product and its connection
with iterated function systems (IFS) theory received great
attention in the last few years (see, for example, \cite{MR3695409} for
skew products involving diffeomorphisms in manifolds, \cite{MR1999572} for
non linear Baker maps and \cite{MR1862809}, \cite{MR2201152},
\cite{MR3806667} for the linear
ones).\\

Analyzing the results in \cite{MR2201152}, in the present paper we
address three main questions:
\begin{itemize}
\item[a)]What happens when we increase the number of skew maps, that is, if
    we consider a skew product IFS on the cylinder $X\times \mathbb{R}$
    formed by maps $G_{j}(x,y)= (T(x) , A_{j}(x)+\lambda y)$ for a finite
    family of $m \geq 2$ potential functions $\{A_{j}\}_{0\leq j\leq m-1}$?
    Do we still have an attractor? How do we characterize
    it?
\item[b)] The second question is about  the
    discounting limit, that is, for $\lambda \to 1$. Can
    we recover some information via ergodic optimization for a finite
    family of potential functions $\{A_{j}\}_{0\leq j\leq m-1}$?
\item[c)] The third question is about the existence and properties
    of the SRB measure for such systems.
\end{itemize}

In \cite{MR3806667}, question (b) was studied and a partial answer to the
conjecture on the structure of the boundary of the attractor was given. The
authors prove that the  superior boundary  of the
attractor is a piecewise differentiable graph $(x, v_{\lambda}(x))$ under
some hypothesis. Other important result was the description of the relation between the upper boundary of the attractor and the
maximizing measures for the potential. This led us to believe that we
could extend these ideas for the IFS problem.

We point out that we choose a simple setting in order to make it easier
to understand the ideas. However, several generalizations are trivially
derived from our arguments. For instance, in the skew IFS problem we
can replace $\textbf{S}^1$ by the $n$-torus $\mathbb{T}^n$ (or other suitable
compact metric space) and consider the skew IFS in $\mathbb{T}^n \times
\mathbb{R}$, or to keep $\textbf{S}^1$ and consider the map $T$ as the
multiplication by $2\leq l \in \mathbb{N}$ or even consider an expanding
automorphism of degree $l \geq 2$. In this case, we simply change
the Lebesgue measure $\ell$ by some $T$-ergodic measure, obtaining the same
results for random SRB measures. We can also easily replace the IFS
$\tau_i(x)=\frac{1}{2}x+ \frac{i}{2}$, $i\in \mathcal{I}=\{0,1\}$, by the
inverse branches of an expanding automorphism $T(x)$, or even a continuous IFS
$R =(X, \tau_{a})_{a \in \mathcal{I}}$ for a compact set $\mathcal{I}$ of
maps on a compact metric space $X$. The results related to IFS
ergodic optimization are still true but we lose the connection with the
invariant set of the skew IFS because the IFS maps are not inverse branches
of the map $T$ that defines the skew IFS.

\subsection{Paper organization}
The structure of the paper is as follows.\\

In Section~\ref{The skew product IFS} we consider the skew IFS problem
showing the existence of a fractal compact invariant set $\Lambda$ and
characterize its boundary using dynamic programming. We also study the random
SRB measure associated to this system, obtaining our main result of the
section, Theorem~\ref{a finite family of
sbr} relating the temporal averages with respect to the skew IFS and the
spatial averages with respect to this measure.

In Section~\ref{IFS ergodic optimization problem for a finite family of
potentials}, we analyze the IFS ergodic optimization problem for a finite
family of potentials. Initially, for ergodic averages, linking this problem
with the superior boundary of $\Lambda$, and after for
$\lambda$-discounted ergodic averages. Additionally, we observe that we can
approximate solutions of the IFS ergodic optimization problem via discounting
limit, that is, for $\lambda \to 1$.

In Section~\ref{appendix} we present a duality result regarding the
discounted holonomic measures for an IFS, giving some additional insight on
the problems of Section~\ref{IFS ergodic optimization problem for a finite
family of potentials}. In particular, in Theorem~\ref{discounted holonomic
prob properties}, we apply this theorem to obtain a paramount
result on the critical value $m_{\lambda}(R)$ of the ergodic optimization problem for
discounted holonomic measures with a trace $\nu$ which is
Theorem~\ref{discounted holonomic prob properties} (b),
$\displaystyle m_{\lambda}(R) =  (1-\lambda) \int v_{\lambda}(x)d\nu(x)$ where $v_{\lambda}(x)$ is the solution of the Bellman equation. This formula will allow us to deal with the discount limit $\lambda \to 1$ later.

\section{The skew product IFS}\label{The skew product IFS}
Consider a fixed number $0< \lambda <1$. We set the topology in $X \times \mathbb{R}$ by considering it as a complete metric space with the distance induced by the quotient  $\mathbb{R}/\mathbb{Z} \times \mathbb{R}$ over $\mathbb{R}^{2}$. Therefore $X \times \mathbb{R}$ is a complete metric space.

Any shift space with $N$ symbols $\{q_0, ..., q_{N-1}\}^{\mathbb{N}}$  used here will be endowed with the product topology induced by the distance $d(\bar{a}, \bar{b})=\lambda^n,$ where $a_k=b_k$ for $k=0, ..., n-1$ and $a_n\neq b_n$.
The shift map $\sigma$ and the concatenation map $*$ are given by
$\sigma(\bar{b})= (b_1, b_2,  ... ) \text{ and } e*\bar{b}=(e, b_0, b_1,  ... ),$
where $\bar{b}=(b_0, b_1,  ... ) \in \{q_0, ..., q_{N-1}\}^{\mathbb{N}}$ (we will denote $\bar{b}$ the sequence and $b_n$ its elements).

Consider $A_c: X \to \mathbb{R}$ for $c \in \mathcal{C}=\{0, ...,m-1\}$ a family of Lipschitz potentials
$$\displaystyle{\rm Lip}(A_c)=\sup_{x\neq y}\frac{|A_c(x)-A_c(y)|}{d(x,y)} < \infty, $$
$T(x)=2 x$, the multiplication on $X=\mathbb{R}/\mathbb{Z}$ and the IFS $R= (X\times\mathbb{R}, G_c(x,y))_{c \in \mathcal{C}}$ where, $G_c(x,y)= (T(x) , A_c (x)+\lambda y).$ Our purpose here is to investigate dynamical and ergodic properties of the skew IFS $R$ in $X \times \mathbb{R}$.

\subsection{The invariant set $\Lambda$}
We may ask if there exists some invariant set or some attractor for the IFS $R$ and how this set can be characterized. Since each map $G_c(x,y)$ expand by a factor 2 in the $x$ direction and contract by a factor $\lambda$ in the $y$ direction, it is not possible to employ the classical methods for contractive IFS. In the case $m=1$, \cite{MR2201152} shows that there exists a solenoidal attractor and study its topological properties.

The orbit of a point $(x',y')$ by the IFS $R$, controlled by the sequence $\bar{c}=(c_0, c_1,  ... ) \in \mathcal{C}^{\mathbb{N}}$, is the projection on $X\times\mathbb{R}$ of the  orbit of the skew dynamical system $$ G (x',y',\bar{c})=(G_{c_{0}}(x',y') , \sigma(\bar{c}))=(T(x') , A_{c_0} (x')+\lambda y', (c_1, c_2,  ... )).$$
The $n$-th iterate $ G ^n(x',y',\bar{c})$ will be $(G_{c_{n-1}}(\cdots(G_{c_{0}}(x',y'))) , \sigma^n(\bar{c}))$ or, more precisely
$$\left(T^n(x') , \sum_{i=0}^{n-1}\lambda^{i}A_{c_{n-1-i}} (T^{n-1-i}(x')) +\lambda^n y', (c_n, c_{n+1},  ... )\right).$$
Therefore, the projection $(x,y)$ on $X\times\mathbb{R}$ is
$x=T^n(x') \text{ and } y=\sum_{i=0}^{n-1}\lambda^{i}A_{c_{n-1-i}} (T^{n-1-i}(x')) +\lambda^n y'.$

\begin{remark}
  We denote ${\rm Per}_{n}(T)=\{x \in X \; | \; T^n (x)=x\}$ the set of periodic points of $T$ with period $n$. Using the formulas for the $n$-th iterate $ G ^n(x',y',\bar{c})$ we can see that each Baker map $G_{c}$ has periodic points projecting in ${\rm Per}_{n}(T)$, that is,
  $${\rm Per}_{n}(G_{c})=\left\{(x',y') \in X\times\mathbb{R} \; | \;x' \in {\rm Per}_{n}(T) \text{ and } y'=\frac{1}{1-\lambda^n}\sum_{i=0}^{n-1}\lambda^{i}A_{c} (T^{n-1-i}(x')) \right\}.$$
  In particular, $|y'| \leq \frac{1}{1-\lambda} \|A_{c}\|_{\infty}$ for any $(x',y') \in {\rm Per}_{n}(G_{c})$.
\end{remark}

We recall that the Hutchinson-Barnsley (HB) operator is given by $\displaystyle F(U)=\bigcup_{c \in \mathcal{C}} G_c(U)$ for any $U \in 2^{X\times\mathbb{R}}$. Considering $\mathcal{K}^{*}(X\times\mathbb{R})$ the family of all nonempty
compact subsets of $X\times\mathbb{R}$, we endow it with the Hausdorff metric and restrict $F$ to this family.

A compact set $\Lambda$ is self-similar (invariant or fractal)  if $F(\Lambda)= \Lambda$. We said that $\Lambda$ is an attractor if there exists $U_{0}$ a neighborhood of $\Lambda$, called the basin of attraction, such that, for any $U \in \mathcal{K}^{*}( U_{0})$ we have $F^n(U) \to\Lambda$. When   $U_{0}=X\times\mathbb{R}$ we say that the attractor is a global attractor.
From \cite{MR2818686}, Lemma 2, if $\Lambda$ is an  attractor with   basin of attraction $U_{0}$, then $\displaystyle\Lambda=\lim_{N \to \infty} \overline{\bigcup_{n\geq N}F^{n}(U)}$  for any $U \in \mathcal{K}^{*}( U_{0})$.

\begin{theorem}\label{globalbasin}
  Consider any $\displaystyle T_0 > \frac{1}{1-\lambda} \max_{c \in \mathcal{C}}\|A_{c}\|_{\infty}$ and the annulus $U_0=X\times (-T_0, T_0)$ then\\
a) $F(\overline{U_0})\subset U_0$.\\
b) $U_0$ is globally attracting, that is, for any $M>0$ we have $F^n (X\times [-M, M]) \subset U_0$ for some $n \in \mathbb{N}$ depending only on $M$.\\
c) If $\Lambda \in \mathcal{K}^{*}(X\times\mathbb{R})$ is self similar, then $ \Lambda \subset U_0$.
\end{theorem}
\begin{proof}
(a)  We know that
$G_{c_{0}}(x',y') =(T(x') , A_{c_0} (x')+\lambda y')=(x,y)$ so
$$|y| \leq |A_{c_0} (x')|+\lambda |y'| \leq \max_{c \in \mathcal{C}}\|A_{c}\|_{\infty} +\lambda |y'|=$$ $$= (1-\lambda) \frac{1}{1-\lambda}\max_{c \in \mathcal{C}}\|A_{c}\|_{\infty} +\lambda |y'|<  (1-\lambda) T_0 +\lambda |y'|\leq (1-\lambda) T_0 +\lambda T_0= T_0.$$
If $(x',y') \in \overline{U_0}$ then  $G_{c_{0}}(x',y') \in U_0$ for any $c_0 \in \mathcal{C}$, thus $F(\overline{U_0})\subset U_0$, because it is a finite union.\\

(b) For a fixed $(x',y') \in X\times\mathbb{R}$ and for any $\bar{c}=(c_0, c_1,  ... ) \in \mathcal{C}^{\mathbb{N}}$ the projection on $X\times\mathbb{R}$ of  $ G ^n (x',y',\bar{c})$  will be in $U_0$ for $n$ big enough.
We need to prove that $\displaystyle \left|\sum_{i=0}^{n-1}\lambda^{i}A_{c_{n-1-i}} (T^{n-1-i}(x')) +\lambda^n y'\right|< T_0,$
when $n$ is big enough uniformly in $\bar{c}$. Indeed,
$$\left|\sum_{i=0}^{n-1}\lambda^{i}A_{c_{n-1-i}} (T^{n-1-i}(x')) +\lambda^n y'\right| \leq \frac{1-\lambda^n}{1-\lambda}\max_{c \in \mathcal{C}}\|A_{c}\|_{\infty} +\lambda^n |y'|$$
$$\leq  \frac{1}{1-\lambda}\max_{c \in \mathcal{C}}\|A_{c}\|_{\infty} +\lambda^n \left(|y'| - \frac{1}{1-\lambda}\max_{c \in \mathcal{C}}\|A_{c}\|_{\infty}\right) <T_0,$$
because $\displaystyle \lambda^n \left(|y'| - \frac{1}{1-\lambda}\max_{c \in \mathcal{C}}\|A_{c}\|_{\infty}\right) \stackrel{n \to \infty}{\to} 0$, $\displaystyle T_0 > \frac{1}{1-\lambda} \max_{c \in \mathcal{C}}\|A_{c}\|_{\infty}$ and this limit is independent of $\bar{c}$,  depending only on $|y'|$. From this, we can conclude that for any $M>0$ we have $F^n (X\times [-M, M]) \subset U_0$, for some $n \in \mathbb{N}$, depending only on $M$.\\

(c) From (b), taking $n$ large enough, we get that $\Lambda$ should be a subset of the ``attracting basin"   $U_0$ because $\Lambda=F^n(\Lambda)$.
\end{proof}

To characterize the points of $\Lambda$, we need to evaluate the images of arbitrarily large order of a point $(x', y') \in X \times \mathbb{R}$ because $F^n (\Lambda)= \Lambda$, for all $n \in \mathbb{N}$, when $\Lambda$ is self similar.
We can write $F^n (\Lambda)$ as
$$\{(x,y)=G_{c_{n-1}}(\cdots(G_{c_{0}}(x',y'))) \; | \; c_{0}, ..., c_{n-1} \in \mathcal{C}, \; (x',y')\in  \Lambda\}=$$
$$\left\{\left(T^n(x'), \sum_{i=0}^{n-1}\lambda^{i}A_{c_{n-1-i}} (T^{n-1-i}(x')) +\lambda^n y'\right) \; | \; c_{0}, ..., c_{n-1} \in \mathcal{C}, \; (x',y')\in  \Lambda\right\}.$$
Since $X$ is compact it is natural to assume, by analogy with \cite{MR2201152}, that the $y$ coordinate is bounded in every infinite pre-orbit of $\Lambda$, so we define
$$\Lambda=\left\{(x,y) \in X\times\mathbb{R} \; | \;\exists M>0, \;\exists \bar{c} \in \mathcal{C}^{\mathbb{N}}, \; \left|  y_n \right|\leq M, \text{ when }  G ^n(x_n,y_n,\sigma^{n}\bar{c})=(x,y),\; \forall n \in \mathbb{N}\right\}.$$

Consider $\mathcal{I}=\{0,1\}$,  $\bar{a}=(a_0, a_1,  ... ) \in  \mathcal{I}^{\mathbb{N}}$ and $\tau_i(x)=\frac{1}{2}x+ \frac{i}{2}$, $i\in \mathcal{I}$, the inverse branches of $T(x)$. We also define $\tau_{i,\bar{a}}(x)= \tau_{a_{i-1}} \circ \tau_{a_{i-2}} ...\circ\tau_{a_0} (x )$.

Inspired by \cite{MR1862809} equation (2), we introduce the function $S: X\times \mathcal{C}^{\mathbb{N}} \times \mathcal{I}^{\mathbb{N}}\to \mathbb{R}$  by
$\displaystyle S_{x}(\bar{c},\bar{a})= \sum_{i=0}^{\infty} \lambda^i A_{c_{i}}(\tau_{i,\bar{a}}(x))$ to study the iteration of points in $\Lambda$.

\begin{proposition}\label{representationlambda}
   The set $\Lambda$ can be written as
   $\displaystyle \Lambda=\left\{\left(x,S_{x}(\bar{c},\bar{a})\right) \in X\times\mathbb{R} \; | \;\forall \bar{c} \in \mathcal{C}^{\mathbb{N}},  \;\forall \bar{a} \in  \mathcal{I}^{\mathbb{N}}\right\}.$
\end{proposition}
\begin{proof}
Let us compute some pre-images of a point $(x,y)$:\\
If $G_{c_{0}}(x_1,y_1)=(x, y)$, then $T(x_1)=x$ and $A_{c_{0}}(x_1)+\lambda y_1= y$. Thus, $\exists a_0 \in \mathcal{I}$, such that, $x_1=\tau_{a_0}(x)$ and $y= A_{c_{0}}(\tau_{a_0}(x)) +\lambda y_1$. Therefore
$y_1=\frac{1}{\lambda}\left(y- A_{c_{0}}(\tau_{a_0}(x))\right).$

Following this procedure we obtain that if $G_{c_{1}}(x_2,y_2)=(x_1,y_1)$, then
$x_2=\tau_{a_1}\tau_{a_0}(x)$ and
$y_1= A_{c_{1}}(x_2) +\lambda y_2.$ So
$y_2=\frac{1}{\lambda^2}\left(y- \left(A_{c_{0}}(\tau_{a_0}(x))- \lambda A_{c_{1}}(\tau_{a_1}\tau_{a_0}(x))\right)\right).$
By proceeding a formal induction we obtain
$x_n=\tau_{a_{n-1}} \circ ... \circ \tau_{a_0}(x)$ and
$y_n=\frac{1}{\lambda^n}\left(y- \sum_{i=0}^{n-1} \lambda^i A_{c_{i}}(\tau_{i,\bar{a}}(x))\right).$

From this computations we conclude that
$$\Lambda=\left\{(x,y) \in X\times\mathbb{R} \; | \;\forall \bar{c} \in \mathcal{C}^{\mathbb{N}}, \; \exists M>0, \; \exists \bar{a} \in  \mathcal{I}^{\mathbb{N}}, \; \left|y- \sum_{i=0}^{n-1} \lambda^i A_{c_{i}}(\tau_{i,\bar{a}}(x))\right| \leq M \lambda^n,\; \forall n \in \mathbb{N}\right\}.$$
Taking $n \to \infty$, we obtain that
$\Lambda=\left\{\left(x,\sum_{i=0}^{\infty} \lambda^i A_{c_{i}}(\tau_{i,\bar{a}}(x))\right) \in X\times\mathbb{R} \; | \;\forall \bar{c} \in \mathcal{C}^{\mathbb{N}},  \;\forall \bar{a} \in  \mathcal{I}^{\mathbb{N}}\right\}.$
\end{proof}

In the next proposition we show that $\Lambda$ is self similar.

\begin{proposition}\label{selfsimilarlambda}
  The set $\Lambda$ is self similar, that is, $F(\Lambda)= \Lambda$.
\end{proposition}
\begin{proof}
Indeed, to show that  $F(\Lambda)\subseteq \Lambda$ we take $(x',y') \in \Lambda$ and $\displaystyle \bar{c} \in \mathcal{C}^{\mathbb{N}}$, $M>0$, $\bar{a} \in  \mathcal{I}^{\mathbb{N}}$, such that, $\displaystyle \left|y'- \sum_{i=0}^{n-1} \lambda^i A_{c_{i}}(\tau_{i,\bar{a}}(x'))\right| \leq M \lambda^n$, for all $n \in \mathbb{N}$. We need to show that $G_b(x',y') \in \Lambda$, for any $b \in \mathcal{C}$.

If $G_{b}(x',y')=(x, y)$, then $T(x')=x$ and $A_{b}(x')+\lambda y'= y$. Thus, $\exists e \in \mathcal{I}$, such that, $x'=\tau_{e}(x)$ and $y= A_{b}(\tau_{e}(x)) +\lambda y'$, therefore
$y'=\frac{1}{\lambda}\left(y- A_{b}(\tau_{e}(x))\right).$

Substituting $x'$ and $y'$ in $\left|y'- \sum_{i=0}^{n-1} \lambda^i A_{c_{i}}(\tau_{i,\bar{a}}(x'))\right| \leq M \lambda^n$,  we obtain
$$\left|\frac{1}{\lambda}\left(y- A_{b}(\tau_{e}(x))\right)- \sum_{i=0}^{n-1} \lambda^i A_{c_{i}}(\tau_{i,\bar{a}}(\tau_{e}(x)))\right| \leq M \lambda^n,$$
which is equivalent to
$\displaystyle\left|y- \sum_{i=0}^{n} \lambda^i A_{b_{i}}(\tau_{i,\bar{e}}(x))\right| \leq M \lambda^{n+1},$
where $\bar{b}=b*\bar{c}=(b, c_0, c_1,...)$ and $\bar{e}=e*\bar{a}=(e, a_0, a_1, ...)$.  Thus, $(x,y) \in \Lambda$.

To prove the opposite inequality, $\Lambda \subseteq F(\Lambda)$, we consider $\displaystyle \left(x,\sum_{i=0}^{\infty} \lambda^i A_{c_{i}}(\tau_{i,\bar{a}}(x))\right) \in \Lambda$ and $x'=\tau_{a_0}(x)$. Then, $x=T(x')$, $x'=\tau_{a_0}(T(x'))$ and
$\displaystyle  \left(x,\sum_{i=0}^{\infty} \lambda^i A_{c_{i}}(\tau_{i,\bar{a}}(x))\right)= G_{c_0}\left(x', \sum_{i=0}^{\infty} \lambda^i A_{c_{i+1}}(\tau_{i,\sigma\bar{a}}(x'))\right).$
This means that $(x, y) \in F(\Lambda)$ because $\displaystyle\left(x', \sum_{i=0}^{\infty} \lambda^i A_{c_{i+1}}(\tau_{i,\sigma\bar{a}}(x'))\right) \in\Lambda$.
\end{proof}

\begin{lemma}\label{continuitySlambda}
  The function $S: X\times \mathcal{C}^{\mathbb{N}} \times \mathcal{I}^{\mathbb{N}}\to \mathbb{R}$ given by
$\displaystyle S_{x}(\bar{c},\bar{a})= \sum_{i=0}^{\infty} \lambda^i A_{c_{i}}(\tau_{i,\bar{a}}(x))$
is Lipschitz continuous, more precisely
$$|S_{x}(\bar{c},\bar{a})-S_{x'}(\bar{b},\bar{e})| \leq \frac{2}{2-\lambda} \max_{c \in \mathcal{C}} {\rm Lip}(A_c) d(x,x') +  \frac{2}{1-\lambda} \max_{c \in \mathcal{C}} \|A_c\|_{\infty} \left(d(\bar{c},\bar{b})+ d(\bar{a},\bar{e})\right),$$
where each distance is taken in the respective space.
\end{lemma}

\begin{proposition}\label{closenesslambda}
  $\Lambda$ is closed, and in particular it is a compact set.
\end{proposition}
\begin{proof}
To see that $\overline{\Lambda}= \Lambda$ we take $(x^k, S_{x^k}(\bar{c}^k,\bar{a}^k)) \to (x,y)$, when $k \to \infty$. We notice that $\mathcal{C}^{\mathbb{N}} \times \mathcal{I}^{\mathbb{N}}$ is a compact space thus we can assume that $(\bar{c}^k,\bar{a}^k) \to (\bar{b},\bar{e})$, possibly by taking a subsequence. We claim that $y= S_{x}(\bar{b},\bar{e})$. Indeed, from  Lemma~\ref{continuitySlambda} we know that
$S_{x}(\bar{b},\bar{e})= \lim_{k\to\infty} S_{x^k}(\bar{c}^k,\bar{a}^k) =y.$ Thus, $(x,y)\in \Lambda$. The compactness follow from the fact that $\Lambda$ is bounded. 
\end{proof}

\subsection{Dynamic programming and the boundary of $\Lambda$} \label{dynprogIFS}
The notation and the main results in dynamic programming presented here are from \cite{MR3248091} and the results on discounted limits are from \cite{Cioletti_2019}. We consider a decision-making process $S=\{X, A, \psi, f, u, \delta\}$ given by:
\begin{itemize}
\item[a)] the state space $(X=\textbf{S}^{1}, d)$ is a compact metric space;
\item[b)] the action space $(\mathcal{C}\times \mathcal{I}, d_{\mathcal{C}\times \mathcal{I}})$ is a compact metric space where $\mathcal{C}$(describe the maps on the IFS) and $\mathcal{I}$ (describe the  injective domains of $T$) are  both finite sets;
\item[c)]the action function $\psi(x)=\mathcal{C}\times \mathcal{I}, \; \forall x \in X$;
\item[d)]the dynamics is given by the contractive IFS $f(x,c,a)=\tau_{a}(x)$ for $a \in \mathcal{I}$ (the dynamics does not depends on $c$);
\item[e)]the immediate return is $u(x,c,a)= A_{c}(\tau_{a}(x)), \; c \in \mathcal{C}, \; a \in \mathcal{I}$; where $A_{c}$ is Lipschitz with respect to $x$, uniformly in $c$;
\item[f)] the discount function is linear, $\delta(t)= \lambda t$ for $0< \lambda <1$.
\end{itemize}
Assuming such hypothesis we can show  that for each fixed $0< \lambda <1$ there exists an unique Lipschitz continuous function $v_\lambda^+$  which satisfies the Bellman equation
$\displaystyle  v_\lambda^+(x)=\sup_{(c,a) \in \mathcal{C}\times \mathcal{I}} A_{c}(\tau_{a}(x)) +  \lambda v_\lambda^+(\tau_{a}(x)),$
where $\displaystyle v_\lambda^+(x)=\max_{\bar{c} \in \mathcal{C}^{\mathbb{N}},\; \bar{a} \in \mathcal{I}^{\mathbb{N}}} \sum_{i=0}^{\infty} \lambda^i A_{c_{i}}(\tau_{i,\bar{a}} x)$ is attained for some optimal pair $(\bar{c}, \bar{a})$. The same is true for $\displaystyle  v_\lambda^-(x)=\min_{\bar{c} \in \mathcal{C}^{\mathbb{N}},\; \bar{a} \in \mathcal{I}^{\mathbb{N}}} \sum_{i=0}^{\infty} \lambda^i A_{c_{i}}(\tau_{i,\bar{a}} x)$ but we are interested in maximization thus, we will use just the notation $v_\lambda^+$ and denote $v_{\lambda}$ if there is no risk of misunderstanding.

One can show that $\displaystyle (1-\lambda)\max_{x}v_{\lambda} \to \bar{u}$  and  $\displaystyle v_{\lambda}(x)-\max_{x}v_{\lambda}$  converges up to subsequence, to a continuous function $b$, such that, $\displaystyle b(x)= \max_{(c,a) \in \mathcal{C}\times \mathcal{I}} A_{c}(\tau_{a}(x)) - \bar{u} + b(\tau_{a}(x))$  that  can be rewritten as the calibrated sub-action equation in ergodic optimization (see \cite{MR2422016} and \cite{MR2563132}) or  the analogous of Hamilton-Jacobi's equation for a discrete Lagrangian dynamics (see \cite{MR2422016}, \cite{MR2283105} and \cite{MR2128794})
$$ \bar{u}= \max_{(c,a) \in \mathcal{C}\times \mathcal{I}} A_{c}(\tau_{a}(x)) +  b(\tau_{a}(x)) -b(x)=\max_{(c,a) \in \mathcal{C}\times \mathcal{I}} d_{x}b(a) + A_{c}(\tau_{a}(x)),$$
where $d_{x}b(a)= b(f(x,c,a)) -b(x)= b(\tau_a x) -b(x)$ is the discrete differential. $\bar{u}$ is uniquely determined by $\displaystyle \bar{u}=\displaystyle\sup_{\mu \in \mathcal{H}} \int  A_{c}(\tau_{a}(x)) d\mu(x,c,a),$ where the set $\mathcal{H}$ of holonomic probabilities is $$\mathcal{H}=\left\{ \mu \in Prob(X\times \mathcal{C}\times \mathcal{I}) \; | \; \int d_{x} g (a) d\mu(x,c,a)=0, \; \forall g\in C^0(X,\mathbb{R}) \right\}.$$

The original idea of considering the holonomic measures as tool to study minimizing orbits in the sense of the Lagrangian dynamics is due to Ma\~n\'e, see  \cite{MR1384478}. The purpose was to use larger and more flexible set of measures instead of using the invariant ones. Then, one must use variational properties of such measures to show that the minimizing measures are actually invariant with respect to the Euler-lagrange flow. In the next decades this concept was successfully adapted for discrete Aubry-Mather theory and IFS theory.

\begin{remark}
  We point out that, in the case where the potential $A$ is not changing at each iteration, the set $\mathcal{C}$ is a single point then, $A_{c}(x)=A(x)$ and $\bar{u}=\displaystyle\sup_{\mu \in \mathcal{H}} \int  A_{c}(\tau_{a}(x)) d\mu(x,c,a)=  \sup_{\mu \in \mathcal{H}} \int  d_x A(a) + A(x) d\mu(x,c,a)=\sup_{\mu \in \mathcal{H}} \int  A(x) d\mu(x,c,a)$, because $A_{c}(\tau_{a}(x)) =    A(\tau_{a}(x))  =   A(\tau_{a}(x)) - A(x) + A(x) =  d_x A(a) + A(x)$ and $\int  d_x A(a)  d\mu(x,c,a)=0$. Thus, we recover exactly the classical setting of ergodic optimization for an IFS as in \cite{MR3806667}.
\end{remark}

The next proposition characterizes the boundary of  $\Lambda$.

\begin{proposition}\label{lowerupperboundary}
 For each $(x,y) \in \Lambda$ we have $v_{\lambda}^{-}(x)\leq y\leq v_{\lambda}^{+}(x)$. Moreover, the graphs $\{(x,v_{\lambda}^{-}(x)) \, | \, x \in X\}$ and $\{(x,v_{\lambda}^{+}(x)) \, | \, x \in X\}$ are subsets of $\Lambda$.
\end{proposition}
\begin{proof}
We  determinate the boundaries of $\Lambda$ using dynamic programming.  Evaluating the above equations we can pick at each step an infinite backwards orbit proving that  $(x,v_{\lambda}^{+}(x)) \in \Lambda$. Thus, the graphs of $v_{\lambda}^{+}$ and $v_{\lambda}^{-}$ are subsets of $\Lambda$. As a consequence, $\Lambda \subseteq \{(x,y) \; | \; v_{\lambda}^{-}(x)\leq y\leq v_{\lambda}^{+}(x)\}$, in particular $\Lambda$ will be a Jordan curve only if $v_{\lambda}^{-}(x)= v_{\lambda}^{+}(x)$ for all $x$.
\end{proof}

\subsection{Approximation of $\Lambda$}

As a motivation, we consider a numerical example to provide some insight on what happens when we draw an orbit of $G$.

\begin{example}\label{chaosgametwopotent}
Consider $m =2$, $\lambda=0.48$ and the potentials  $A_0, A_{1}: X \to \mathbb{R}$ given by $A_0(x)=(x-\frac{1}{2})^2$ and $A_1(x)=2x, \; 0\leq x \leq\frac{1}{2} $, $A_1(x)=-2x+2, \; \frac{1}{2}\leq x \leq 1$. Let $R= (X\times\mathbb{R}, G_j(x,y))$ be the IFS given by $G_j(x,y)= (T(x) , A_j (x)+\lambda y)$. In this example we choose an initial point $(x_0, y_0)=(0.2472135954, 0.1)$ and draw its orbit using a chosen sequence $\bar{c}=(c_0, c_1, ...)$.

We consider three different situations:
\begin{itemize}
  \item[a)] When $ \bar{c}=(0, 0, 0, ...)$, then iterate only $G_0(x,y)= (T(x) , A_0 (x)+ 0.48 y)$  beginning in $(x_0, y_0)$.
  \item[b)] When we choose, in a random way, a sequence $\bar{c}=(c_0, c_1, ...) \in \{0,1\}^N$ and iterate $G_{c_i}(x,y)= (T(x) , A_{c_i} (x) + 0.48 y)$  beginning in $(x_0, y_0)$.
  \item[c)] When $\bar{c}=(1, 1, 1, ...)$ and iterate only $G_1(x,y)= (T(x) , A_1 (x)+ 0.48 y)$   beginning in $(x_0, y_0)$.
\end{itemize}
The associated pictures are showed in the Figure~\ref{bb}.
\begin{figure}[!ht]
  \centering
  \includegraphics[width=3cm]{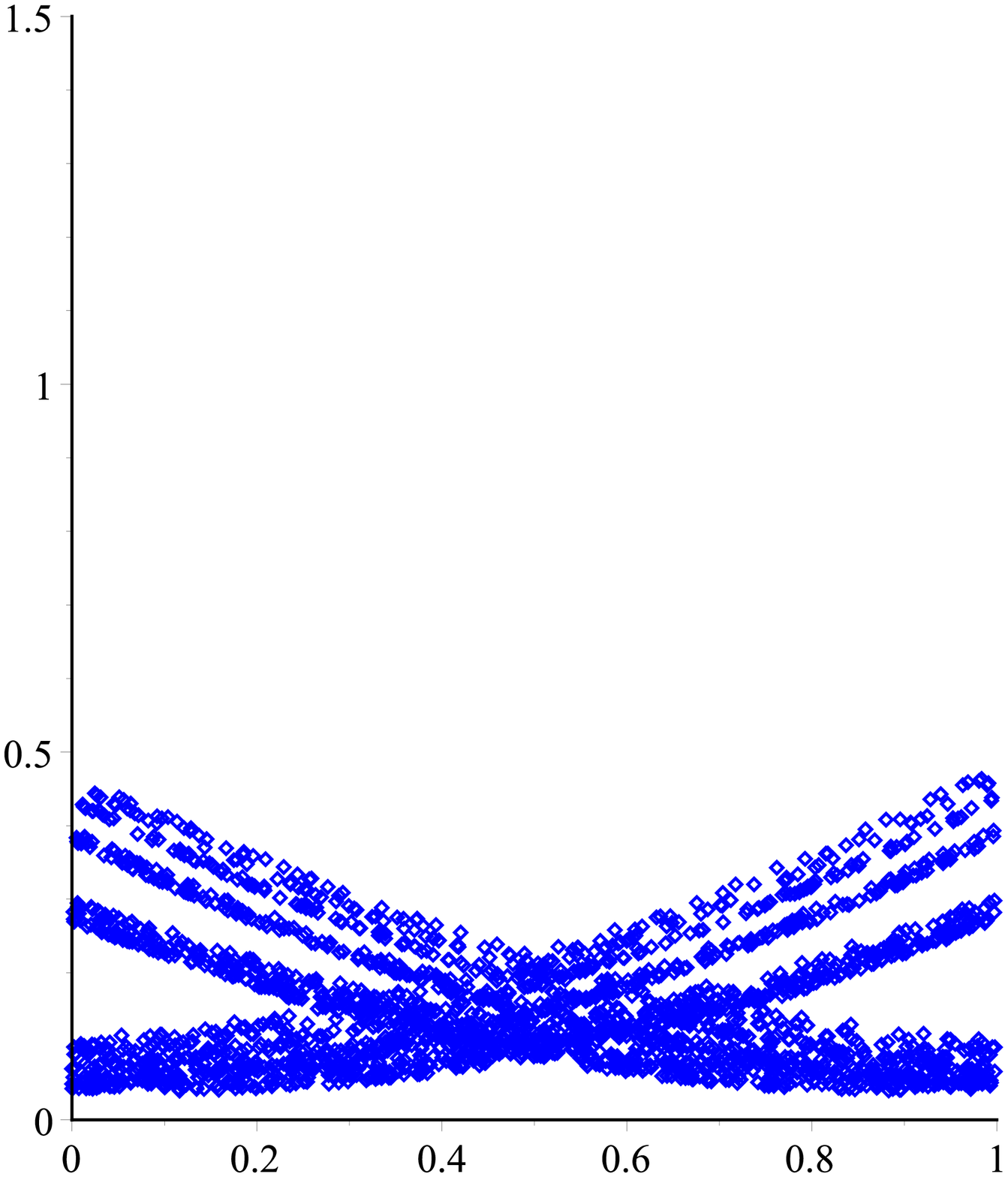} \hspace{1cm}
  \includegraphics[width=3cm]{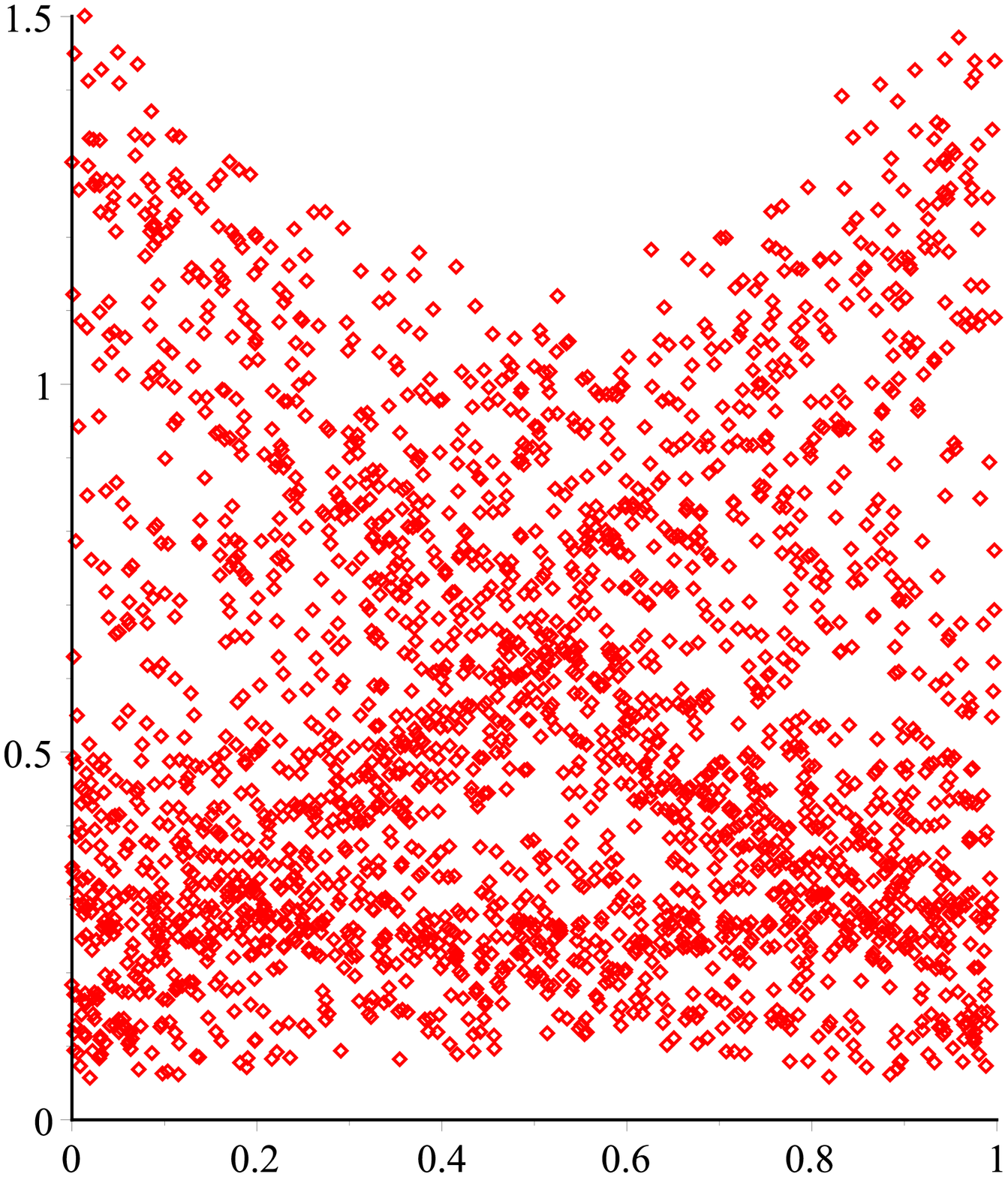}\hspace{1cm}
  \includegraphics[width=3cm]{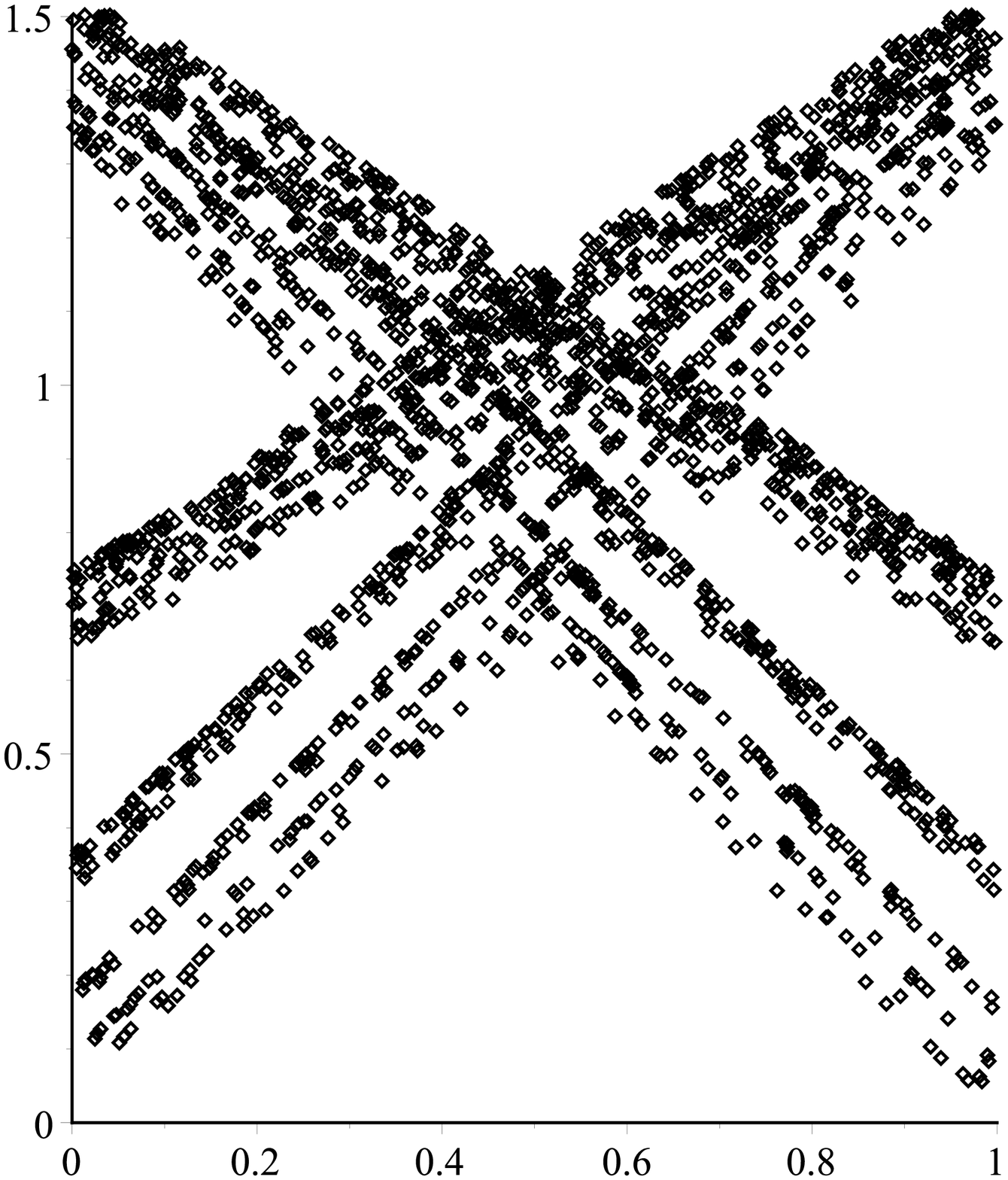}
  \caption{Approximating the attractor: a, b and c, from the left to the right.} \label{bb}
\end{figure}

We can see in (b) that a finite family of potential system(red/center) presents some mixing of the two autonomous (a) and (c) (blue/left and black/right) studied by \cite{MR2201152}. This example makes us to conjecture that closure of a typically orbit of the IFS drawn the picture of the invariant set $\Lambda$.

Finally, in the Figure~\ref{bball} we draw a picture, describing the iteration using a random orbit with 10.000 iterations after the iterate 1.000. We also show the approximations of $v_{0.48}^{+}(x)$(green) and $v_{0.48}^{-}(x)$(yellow) obtained by iteration of the contractive operator  $\displaystyle \mathcal{L}^{+}(f)(x)=\max_{c \in \{0,1\}, a \in \{0,1\}} A_{c }(\tau_{a} x) + 0.48 f(\tau_{a} x),$
(resp. $\mathcal{L}^{-}$) because $\mathcal{L}^{+}(v_{0.48}^{+})=v_{0.48}^{+}$ $\left(\text{resp.}\mathcal{L}^{-}(v_{0.48}^{-})=v_{0.48}^{-}\right)$.
\begin{figure}[!ht]
  \centering
  \includegraphics[width=4cm]{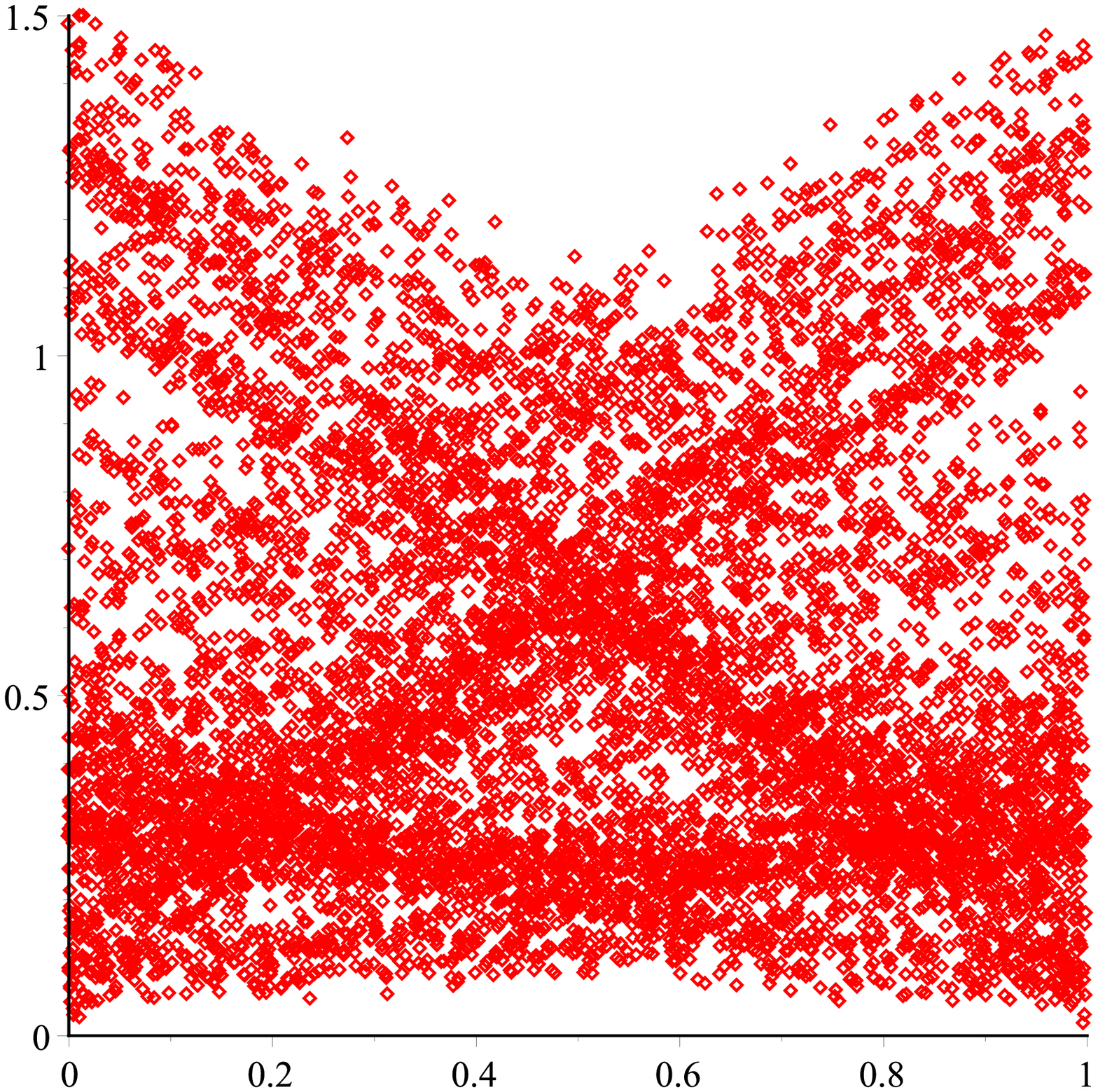}
  \includegraphics[width=4cm]{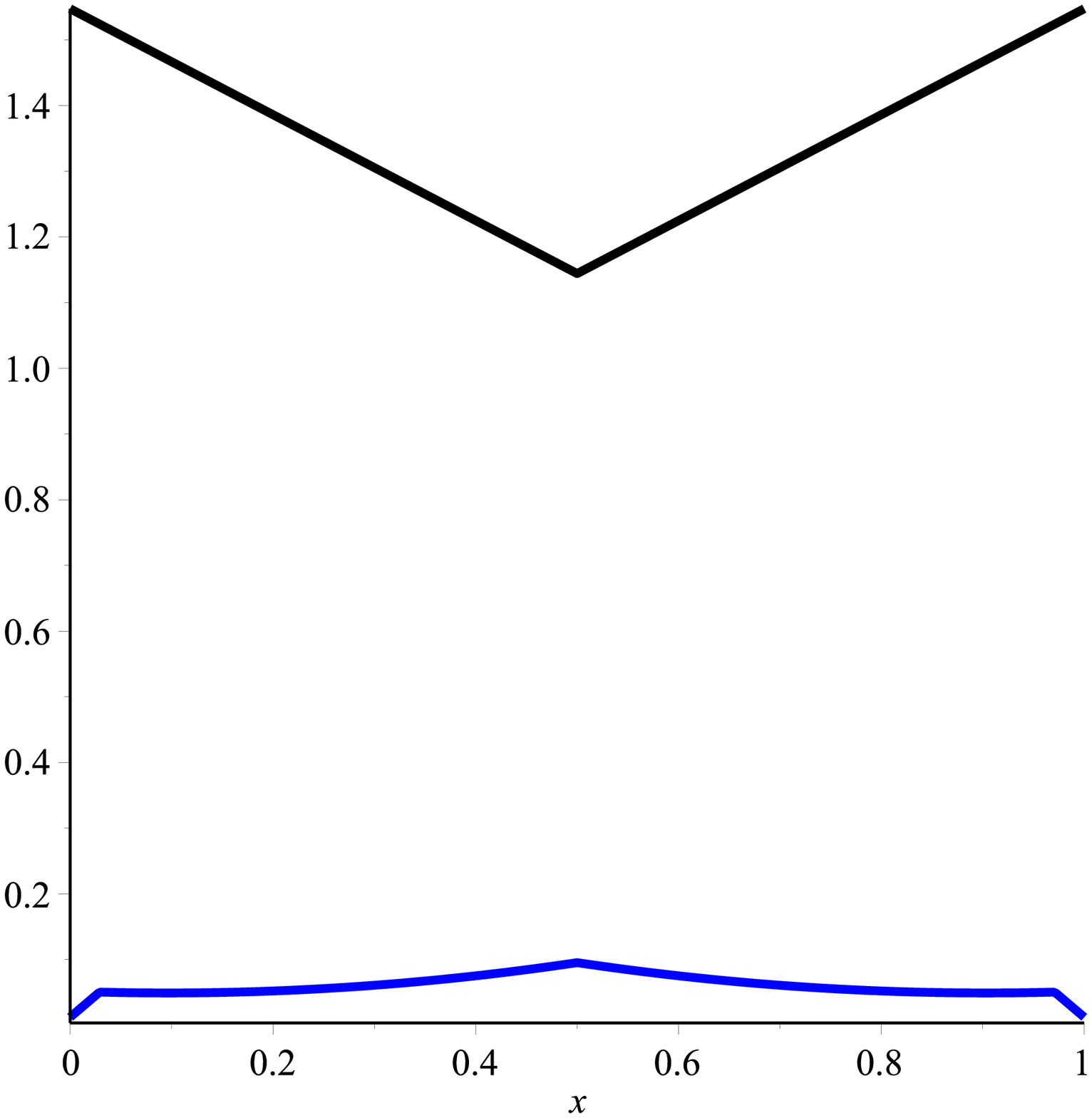}
  \caption{From the left to the right: Approx. of $\Lambda$ and the upper and lower boundary $v_{0.48}^{+}(x)$ and $v_{0.48}^{-}(x)$.} \label{bball}
\end{figure}
\end{example}

In the same setting of the  Example~\ref{chaosgametwopotent}, we consider  a periodic point $x_0=\frac{1}{3}$ and  $(x_0, y_0)=(0.33..., 1.4)$. In the Figure~\ref{cc}, we choose a sequence $\bar{c}=(c_0, c_1, ...) \in \{0,1\}^\mathbb{N}$, iterate $G_{c_i}(x,y)= (T(x) , A_{c_i} (x) + 0.48 y)$ by 2000 times  beginning in $(x_0, y_0)$ and we plot it for $i\geq 500$.\\
  \begin{figure}[!ht]
  \centering
  \includegraphics[width=4cm, height=3cm]{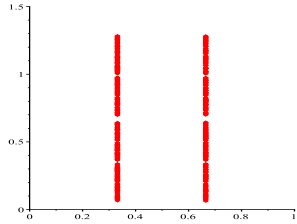}
  \caption{Iteration of the compact $\{(0.33...,1.4)\}$  by $F$.} \label{cc}
\end{figure}

\begin{proposition}\label{Lambdais notattractor}
$\Lambda$ is not an attractor with any basin of attraction $U_0$, in the IFS sense.
\end{proposition}
\begin{proof}
  Consider an arbitrary point $\{(1/3,y')\}$ and calculate $F^n(\{(1/3,y')\})$. A simple computation shows  that
$$G_{b_{-1}}(1/3,y')= (T(1/3) , A_{b_{-1}} (1/3) + \lambda y')= (2/3 , A_{b_{-1}} (1/3) + \lambda y') $$
and
$$G_{b_{-2}}G_{b_{-1}}(1/3,y')=(T(2/3) , A_{b_{-2}} (2/3)+ \lambda A_{b_{-1}} (1/3) + \lambda^2 y')=$$
$$=(1/3 , A_{b_{-2}} (2/3)+ \lambda A_{b_{-1}} (1/3) + \lambda^2 y').$$
The same behavior occurs for $G_{b_{-n}}\cdots G_{b_{-1}}(1/3,y')$ depending only if $n$ is even or odd.

We can see that $\Lambda$ is not an attractor with basin of attraction $U_0$, in the IFS sense, because $\displaystyle  \displaystyle \pi_{x}\left(\lim_{n \to \infty}F^n(\{(1/3,y')\})\right)=\left\{\frac{1}{3}, \frac{2}{3}\right\},$ thus it is impossible to obtain
$\displaystyle \lim_{n \to \infty} F^n(\{(1/3,y')\}) =\Lambda$. This contradicts the fact that $\{(1/3,y')\}$ is a compact set in the basin of attraction $U_0$.
\end{proof}

\subsection{The dynamic in $\Lambda$ and the random SRB measure}

As $\Lambda$ is a compact subset and $F(\Lambda)=\Lambda$ we can restrict the IFS $R= (X\times\mathbb{R}, G_c(x,y))_{c \in \mathcal{C}}$ to $\Lambda$, that is, $R'= (\Lambda, G_c(x,y))_{c \in \mathcal{C}}$.

We recall that
$\Lambda=\left\{\left(x,S_{x}(\bar{c},\bar{a})\right) \in X\times\mathbb{R} \; | \;\forall \bar{c} \in \mathcal{C}^{\mathbb{N}},  \;\forall \bar{a} \in  \mathcal{I}^{\mathbb{N}}\right\},$
where the function $S: X\times \mathcal{C}^{\mathbb{N}} \times \mathcal{I}^{\mathbb{N}}\to \mathbb{R}$ given by $\displaystyle S_{x}(\bar{c},\bar{a})= \sum_{i=0}^{\infty} \lambda^i A_{c_{i}}(\tau_{i,\bar{a}}(x))$
is Lipschitz continuous  by Lemma~\ref{continuitySlambda}.

We introduce the address function $\pi : X \to \mathcal{I}$ given by $\pi(x)=e \in \mathcal{I} \Leftrightarrow \tau_{e}T(x)=x,$
which is well defined for $T$ because, in each point, there exists a left inverse function.

To study the iteration of points in $\Lambda $ we  compute, for each~\footnote{We use a little bit different indexation for $b$ starting from $\bar{b}=(b_{-1}, b_{-2},...)$ instead $\bar{b}=(b_{0}, b_{1}, ...)$. This is made in order to avoid confusion on the concatenation operation $b_{-1}*\bar{c}=(b_{-1},c_0, c_1,...)$.} $\bar{b}=(b_{-1}, b_{-2}, ...)\in \mathcal{C}^{\mathbb{N}}$ the image
$ G \left(x,S_{x}(\bar{c},\bar{a}), \bar{b}\right)=\left(T(x),S_{T(x)}(b_{-1}*\bar{c},e_{-1}*\bar{a}), \sigma\bar{b}\right)$
where $e_{-1}=\pi(x)\in \mathcal{I}$ (the unique element in $\mathcal{I}$, such that, $\tau_{e_{-1}}T(x)=x$), $b_{-1}*\bar{c}=(b_{-1},c_0, c_1,...)$ and $e_{-1}*\bar{a}=(e_{-1},a_0, a_1,...)$.
For an arbitrary $n$ we get $ G ^n \left(x,S_{x}(\bar{c},\bar{a}), \bar{b}\right)$ as
$\left(T^{n}(x),S_{T^{n}(x)}(b_{-n}\cdots b_{-1}*\bar{c},e_{-n}\cdots e_{-1}*\bar{a}), \sigma^{n}\bar{b}\right)$
where $e_{-i}=\pi(T^{i-1}(x)), i=1,...,n$.

This shows that $ G $ has a simple behavior in $\Lambda$. Therefore is natural to consider a conjugation between $ G $ and the symbolic dynamics. To do that, we define the map $\theta_{b}: X\times \mathcal{I}^{\mathbb{N}} \times \mathcal{C}^{\mathbb{N}} \to X\times \mathcal{I}^{\mathbb{N}} \times \mathcal{C}^{\mathbb{N}}$ by $\theta_{b} (x, \bar{a}, \bar{c})= (T(x), \pi(x)*\bar{a}, b*\bar{c}),$ for any $b \in \mathcal{C}$. We also define the map
$$\theta(x, \bar{a}, \bar{c}, \bar{b})= (\theta_{b_{-1}} (x, \bar{c}, \bar{a}), \sigma(\bar{b}))= (T(x), \pi(x)*\bar{a}, b_{-1}*\bar{c}, \sigma(\bar{b}))$$
and the map $\Psi: X \times \mathcal{I}^{\mathbb{N}}\times \mathcal{C}^{\mathbb{N}} \times \mathcal{C}^{\mathbb{N}}\to X\times \mathbb{R} \times \mathcal{C}^{\mathbb{N}}$ by
$\Psi(x, \bar{a}, \bar{c}, \bar{b})= (x, S_{x}(\bar{c},\bar{a}), \bar{b}).$

\begin{lemma}\label{commut}
  Consider the maps $ G , \Psi, \theta$ above defined. Then, the following properties are true:
  \begin{itemize}
    \item[a)] $S_{T(x)}(b*\bar{c},\pi(x)*\bar{a})= A_{b}(x) + \lambda S_{x}(\bar{c},\bar{a})$,for any $b \in \mathcal{C}$.
    \item[b)] $ G \circ \Psi =\Psi \circ \theta$.
  \end{itemize}
\end{lemma}
\begin{proof}
  (a) We already proved that.\\
  (b) To see that, we compute both sides of the equation
  $$ G \circ \Psi (x, \bar{a}, \bar{c}, \bar{b})=  G (x, S_{x}(\bar{c},\bar{a}), \bar{b})=\left(T(x),S_{T(x)}( b_{-1}*\bar{c},\pi(x)*\bar{a}), \sigma\bar{b}\right)=$$
  $$= \Psi (T(x), \pi(x)*\bar{a}, b_{-1}*\bar{c}, \sigma\bar{b})= \Psi \circ \theta (x, \bar{a}, \bar{c}, \bar{b}).$$
\end{proof}

\begin{theorem}\label{a finite family of sbr}
  Consider $ G , \Psi, \theta$ as in Lemma~\ref{commut}. Recalling that $\mathcal{C}=\{0, ...,m-1\}$ and $\mathcal{I}=\{0, 1\}$, let $\nu$ be the uniform Bernoulli measure on $\mathcal{C}^{\mathbb{N}}$ given by $\nu(j)=\frac{1}{m}$, $\eta$ be the  uniform Bernoulli measure on $\mathcal{I}^{\mathbb{N}}$ given by $\eta(i)=\frac{1}{2}$ and $\ell$ be the normalized Lebesgue measure on $X$. The probability measure $\mu$ in $X\times \mathbb{R} \times \mathcal{C}^{\mathbb{N}}$ given by
  $$\int_{X\times \mathbb{R} \times \mathcal{C}^{\mathbb{N}}} g(x,y,\bar{b}) d\mu(x,y,\bar{b}) = \int_{X\times \mathcal{I}^{\mathbb{N}}\times \mathcal{C}^{\mathbb{N}} \times \mathcal{C}^{\mathbb{N}}} g(\Psi(x, \bar{a}, \bar{c}, \bar{b})) d\ell(x)d\eta(\bar{a})d\nu(\bar{c})d\nu(\bar{b})$$
  that is, $\mu= \Psi(\ell\times\eta\times\nu^2)$ is ergodic with respect to $ G $. In particular, there exists a set $\Omega \subseteq X\times \mathbb{R} \times \mathcal{C}^{\mathbb{N}}$  with $\mu(\Omega)=1$,  such that, for all $(x, y, \bar{b}) \in \Omega$,
  $$\frac{1}{N}\sum_{j=0}^{N-1} g( G ^{j}(x,y,\bar{b})) \to \int_{X\times \mathbb{R} \times \mathcal{C}^{\mathbb{N}}} g(x,y,\bar{b}) d\mu(x,y,\bar{b}),$$
  for all $g \in L^1(\mu)$. This measure is called the \textbf{random SRB} measure for the IFS $R= (X\times\mathbb{R}, G_c(x,y))_{c \in \mathcal{C}}$.
\end{theorem}
\begin{proof}
   The proof is a consequence of the Birkhoff theorem if we write our map in a clever way.
   Consider the map $\theta'(x,\bar{a})= (T(x), \pi(x)*\bar{a})$ introduced by \cite{MR1862809}. We notice that $\ell \times \eta$ is ergodic for $\theta'$.
   We also define $\theta''(\bar{c}, \bar{b})= (b_{-1}*\bar{c}, \sigma(\bar{b}))=((b_{-1}, c_0, c_1, ...), (b_{-2}, b_{-3}, ...))$ that is the two sided-shift in $\mathcal{C}^{\mathbb{Z}}$. It is a known fact that $\nu^2$ is the  uniform Bernoulli measure which is ergodic with respect to the bilateral shift in $\mathbb{Z}$.
   As $\theta=\theta' \times \theta''$ is a factor map, the product $\ell \times \eta \times \nu^2$ is ergodic for $\theta$.
  From Lemma~\ref{commut} we can transfer the ergodicity and the invariance of $\theta$ with respect to $\ell \times \eta \times \nu^2$ to $ G $ with respect to $\mu$. Indeed,
  $$\int_{X\times \mathbb{R} \times \mathcal{C}^{\mathbb{N}}} g\circ G (x,y,\bar{b}) d\mu(x,y,\bar{b})=  \int_{X\times \mathcal{I}^{\mathbb{N}}\times \mathcal{C}^{\mathbb{N}} \times \mathcal{C}^{\mathbb{N}}} g\circ G (\Psi(x, \bar{a}, \bar{c}, \bar{b})) d\ell(x)d\eta(\bar{a})d\nu(\bar{c})d\nu(\bar{b})= $$
  $$= \int_{X\times\mathcal{I}^{\mathbb{N}} \times \mathcal{C}^{\mathbb{N}}\times \mathcal{C}^{\mathbb{N}}} g\circ\Psi(\theta(x, \bar{c}, \bar{a}, \bar{b})) d\ell(x)d\eta(\bar{a})d\nu(\bar{c})d\nu(\bar{b})= \int_{X\times \mathbb{R} \times \mathcal{C}^{\mathbb{N}}} g(x,y,\bar{b}) d\mu(x,y,\bar{b}).$$
  Thus, we can apply the ergodic Birkhoff theorem and the result follows.
\end{proof}

\begin{remark}
  We notice that the measure $\mu$ given by Theorem~\ref{a finite family of sbr}, and consequently the set $\Omega$ depends on $\lambda$ because $\mu= \Psi(\ell\times\eta\times\nu^2)$ and $\Psi(x, \bar{a}, \bar{c}, \bar{b})= (x, S_{x}(\bar{c},\bar{a}), \bar{b}).$ When $\lambda \to 1$ the discounted sum $S_{x}(\bar{c},\bar{a})$ does not converges.
\end{remark}

\begin{corollary}\label{SRB average}
Let $\mu$ be the random SRB measure given by Theorem~\ref{a finite family of sbr}. The following statements are true:
\begin{itemize}
  \item[a)] $\displaystyle \int_{X} v_{\lambda}^{-}(x)d\ell(x)\leq \int_{X\times \mathbb{R} \times \mathcal{C}^{\mathbb{N}}} y\; d\mu(x,y,\bar{b})\leq  \int_{X} v_{\lambda}^{+}(x)d\ell(x) .$ \\
  \item[b)] $\Pi_{2}(x,y,\bar{b})= y \in L^{1}(\mu)$.\\
  \item[c)] There exists a set $\Omega \subseteq X\times \mathbb{R} \times \mathcal{C}^{\mathbb{N}}$ ($\mu(\Omega)=1$),  such that, for all $(x, \bar{b}=(b_{-1},b_{-2},...) ) \in \Pi_{\{1,3\}}\Omega$
$$\frac{1}{N}\sum_{j=1}^{N-1}  A_{b_{-j}} (T^{j-1}(x)) \to (1-\lambda) \int_{X\times \mathbb{R} \times \mathcal{C}^{\mathbb{N}}} y \; d\mu(x,y,\bar{b}).$$
  \item[d)] The limit in (c) can be improved by $\forall \varepsilon>0$, $\exists \lambda_0$,  $\forall \lambda > \lambda_0$ exists a set $\Omega_{\lambda} \subseteq X\times \mathbb{R} \times \mathcal{C}^{\mathbb{N}}$ ($\mu(\Omega_{\lambda})=1$),  such that, for all $(x, \bar{b}=(b_{-1},b_{-2},...) ) \in \Pi_{\{1,3\}}\Omega_{\lambda}$, $\exists N_0$, such that,
$$ \frac{1}{N}\sum_{j=1}^{N-1}  A_{b_{-j}} (T^{j-1}(x)) \leq \bar{u} + \varepsilon,$$
for all $N \geq N_0$, where $\displaystyle \bar{u}=\displaystyle\sup_{\mu \in \mathcal{H}} \int  A_{c}(\tau_{a}(x)) d\mu(x,c,a)$.
\end{itemize}

\end{corollary}
\begin{proof}

(a)-(b) We claim that $g \in L^{1}(\mu)$ where  $g(x,y,\bar{b})=\Pi_{2}(x,y,\bar{b})=y$. To see this we recall that
 $$\int_{X\times \mathbb{R} \times \mathcal{C}^{\mathbb{N}}} \Pi_{2}(x,y,\bar{b}) d\mu(x,y,\bar{b})=  \int_{X\times \mathcal{I}^{\mathbb{N}}\times \mathcal{C}^{\mathbb{N}} \times \mathcal{C}^{\mathbb{N}}} \Pi_{2}(\Psi(x, \bar{a}, \bar{c}, \bar{b})) d\ell(x)d\eta(\bar{a})d\nu(\bar{c})d\nu(\bar{b})= $$
 $$= \int_{X\times \mathcal{I}^{\mathbb{N}}\times \mathcal{C}^{\mathbb{N}} \times \mathcal{C}^{\mathbb{N}}} S_{x}(\bar{c},\bar{a}) d\ell(x)d\eta(\bar{a})d\nu(\bar{c})d\nu(\bar{b}) \in \left[\int_{X} v_{\lambda}^{-}(x)d\ell(x), \;  \int_{X} v_{\lambda}^{+}(x)d\ell(x)\right],$$
 because $v_{\lambda}^{-}(x)\leq  S_{x}(\bar{c},\bar{a}) \leq v_{\lambda}^{+}(x)$ from Proposition~\ref{lowerupperboundary}.

 (c) Using Theorem~\ref{a finite family of sbr} we get
 $$\frac{1}{N}\sum_{n=0}^{N-1} \Pi_{2}( G ^{n}(x,y,\bar{b})) \to \int_{X\times \mathbb{R} \times \mathcal{C}^{\mathbb{N}}} \Pi_{2}(x,y,\bar{b}) d\mu(x,y,\bar{b}),$$
where $\Pi_{2}\circ G ^n \left(x,y,\bar{b}\right)=\Pi_{2}\left(T^{n}(x),y^n, \sigma^{n}\bar{b}\right)=y^n.$
We notice that $y^0=y$ and for $n \geq 1$ we have the formula
$\displaystyle y^n = \lambda^n y + \sum_{j=1}^{n}\lambda^{n-j} A_{b_{-j}} (T^{j-1}(x)),$
so
$$\frac{1}{N}\sum_{n=0}^{N-1} \Pi_{2}( G ^{n}(x,y,\bar{b}))= \frac{y}{N} + \frac{1}{N}\sum_{n=1}^{N-1} y^n = \frac{y}{N} +  \frac{1}{N}\sum_{n=1}^{N-1} \left[\lambda^n y + \sum_{j=1}^{n}\lambda^{n-j} A_{b_{-j}} (T^{j-1}(x))\right]=$$
$$= y \frac{1}{N}\sum_{n=0}^{N-1}\lambda^n +  \frac{1}{N}\sum_{n=1}^{N-1} \left[ \sum_{j=1}^{n}\lambda^{n-j} A_{b_{-j}} (T^{j-1}(x))\right].$$
As $\displaystyle y \frac{1}{N}\sum_{n=0}^{N-1}\lambda^n  \to 0$
for any $y$, we obtain
$$\frac{1}{N}\sum_{n=1}^{N-1} \left[ \sum_{j=1}^{n}\lambda^{n-j} A_{b_{-j}} (T^{j-1}(x))\right] \to \int_{X\times \mathbb{R} \times \mathcal{C}^{\mathbb{N}}} y d\mu(x,y,\bar{b}).$$

Rewriting the above sum we obtain
$$\frac{1}{N}\sum_{n=1}^{N-1} \left[ \sum_{j=1}^{n}\lambda^{n-j} A_{b_{-j}} (T^{j-1}(x))\right]=
\frac{1}{N}\sum_{j=1}^{N-1}\left[\sum_{i=0}^{N-1-j}\lambda^{i}\right] A_{b_{-j}} (T^{j-1}(x))=$$
$$=\frac{1}{N}\sum_{j=1}^{N-1}\left[\frac{\lambda^{N-j} -1}{\lambda -1}\right] A_{b_{-j}} (T^{j-1}(x))= $$ $$=\frac{1}{N(\lambda -1)}\sum_{j=1}^{N-1}\lambda^{N-j} A_{b_{-j}} (T^{j-1}(x)) - \frac{1}{N(\lambda -1)}\sum_{j=1}^{N-1}  A_{b_{-j}} (T^{j-1}(x)).$$

We notice that $|A_{b_{-j}} (T^{j-1}(x))| \leq \|A_{c}\|_{\infty}$ thus
$\lim_{N \to \infty}\frac{1}{N(\lambda -1)}\sum_{j=1}^{N-1}\lambda^{N-j} A_{b_{-j}} (T^{j-1}(x))=0.$

Using this, we conclude that
$\frac{1}{N}\sum_{j=1}^{N-1}  A_{b_{-j}} (T^{j-1}(x)) \to (1-\lambda) \int_{X\times \mathbb{R} \times \mathcal{C}^{\mathbb{N}}} y d\mu(x,y,\bar{b}).$

(d) From (c) we obtain a set $\Omega_{\lambda} \subseteq X\times \mathbb{R} \times \mathcal{C}^{\mathbb{N}}$ ($\mu(\Omega_{\lambda})=1$),  such that, for all $(x, \bar{b}=(b_{-1},b_{-2},...) ) \in \Pi_{\{1,3\}}(\Omega_{\lambda})$
$$\frac{1}{N}\sum_{j=1}^{N-1}  A_{b_{-j}} (T^{j-1}(x)) \to (1-\lambda) \int_{X\times \mathbb{R} \times \mathcal{C}^{\mathbb{N}}} y d\mu(x,y,\bar{b}).$$

On the other hand, we know that $(1-\lambda)\max v_{\lambda}^{+} \to \bar{u} $, where $\bar{u}$ is the maximum of the integrals of $(x,c, a) \to A_{c}(\tau_{a}(x))$, for all the holonomic measures over $X\times \mathcal{C} \times \mathcal{I}$. Therefore, for any $\varepsilon>0$ there exists $\lambda_0 \in(0,1)$, such that, for all $\lambda > \lambda_0$ we have $(1-\lambda) \max v_{\lambda}^{+}  \leq \bar{u} +\varepsilon$.

Thanks to the item (a), we notice that
$$\int_{X\times \mathbb{R} \times \mathcal{C}^{\mathbb{N}}} y \, d\mu(x, y, \bar{b})\leq \int_{X} v_{\lambda}^{+}(x)d\ell(x) \leq \max v_{\lambda}^{+}.$$
Multiplying the above inequality by $(1-\lambda)$ and taking limits we get $N_0 \in \mathbb{N}$, such that,
$$\displaystyle\frac{1}{N}\sum_{j=1}^{N-1}  A_{b_{-j}} (T^{j-1}(x)) \leq \bar{u} +\varepsilon,$$
for all $N >N_0$.
\end{proof}

\subsection{The associated ergodic optimization problem}
In ergodic optimization the central problem is to find the maximum of the integral of a potential function $A: X \to \mathbb{R}$ with respect to any probability measure $\mu$ which is invariant under a dynamic $T: X \to X$, that is,
$\displaystyle  m_{T}(A)=\sup_{T(\mu)=\mu} \int_{X} A(x) d\mu(x).$
Obviously, this problem has very different solutions depending strongly on $A$ and on $T$. Moreover, the Birkhoff theorem tells us that invariant measures can be understood using empirical averages, that is,
$\displaystyle \frac{1}{N}\sum_{i=0}^{N-1} A(T^{i}(x)) \to \hat{A}(x)$
a.e. and $\int_{X} \hat{A}(x) d\mu(x)=\int_{X} A(x) d\mu(x)$. In particular, $\hat{A}(x) =\int_{X} A(x) d\mu(x)$ a.e., when $\mu$ is ergodic.

In physical applications or optimal controlling problems these  empirical averages represent evaluations of an observable information along of the dynamical trajectory $x_0=x, x_1=T(x), x_2=T^2(x), ...$. So, it is very natural to suppose that the value $A$ can be influenced by some noise producing a range of possible variations $A_{c}: X \to \mathbb{R}$ for $c \in \mathcal{C}$. In such case, the empirical averages will be controlled by a sequence $(c_0, c_1,...)  \in \mathcal{C}^{\mathbb{N}}$:
$\displaystyle\frac{1}{N}\sum_{i=0}^{N-1} A_{c_{i}}(T^{i}(x)).$

From Theorem~\ref{a finite family of sbr} we can give an immediate solution for this problem for $X=\textbf{S}^{1}$. Is easy to see that if $g(x,y,\bar{b})= A_{b_{-1}}(x) \in L^1(\mu)$, then
there exists a set $\Omega \subseteq X\times \mathbb{R} \times \mathcal{C}^{\mathbb{N}}$,  with $\mu(\Omega)=1$,  such that, for all $(x, y, \bar{b}) \in \Omega$,
$$\frac{1}{N}\sum_{j=1}^{N-1} A_{b_{-j}}(T^{j-1}(x)) \to \int_{X\times \mathbb{R} \times \mathcal{C}^{\mathbb{N}}} A_{b_{-1}}(x) d\mu(x,y,\bar{b})=$$ $$=\int_{\mathcal{C}^{\mathbb{N}}}\left(\int_{X} A_{b_{-1}}(x) d\ell(x)\right) d\nu(\bar{b})=\int_{X} \frac{1}{m}\sum_{c=0}^{m-1} A_{c}(x) d\ell(x).$$

\section{IFS Ergodic optimization for a finite family of potentials} \label{IFS ergodic optimization problem for a finite family of potentials}
In this section we formulate the IFS ergodic optimization problem for a finite family of potentials. Our results naturally generalizes the theory of  IFS ergodic optimization developed in \cite{MR2461833}.

As usual, the endomorphism dynamics $x_0=x, x_1=T(x), x_2=T^2(x), ...$ can be replaced by an IFS dynamics $(X, \tau_{i})_{i \in \mathcal{I}}$, that is, $x_0=x, x_1=\tau_{a_0}(x_0), x_2=\tau_{a_1}(x_1), ..., x_{i+1}=\tau_{a_i}(x_i), ...$. In this way the empirical averages
$\displaystyle\frac{1}{N}\sum_{i=0}^{N-1} A_{c_{i}}(\tau_{a_i}(x_i)),$
will be controlled by a sequence $(c_0,a_0, c_1,a_1...)  \in (\mathcal{C}\times \mathcal{I})^{\mathbb{N}}$.

We want to investigate the variational problem associated to these averages because it is closely related to the behavior of our skew product IFS when the discount parameter $\lambda$ is close to 1. We call this ``\textbf{the IFS ergodic optimization for a finite family of potentials}", $A_{c}: X \to \mathbb{R}$, for $c \in \mathcal{C}$. To do that, we need to define what is the appropriated notion of invariant measure and a variational formulation.

\begin{quote}
\emph{\textbf{Given a finite set of potentials $A_{c}: X \to \mathbb{R}$ for $c \in \mathcal{C}$ and an IFS $R=(X, \tau_a)_{a \in \mathcal{I}}$ we want to find the value
$\displaystyle m(R)=\sup_{\mu \in \mathcal{H}} \int  A_{c}(\tau_a x) d\mu(x,c,a),$
over the set of holonomic probabilities $$\mathcal{H}=\left\{ \mu \in Prob(X\times\mathcal{C}\times\mathcal{I}) \; | \; \int d_{x} g (a) d\mu(x,c,a)=0, \; \forall g\in C^0(X,\mathbb{R})  \right\},$$ and the optimal measures $\mu_{\infty} \in \mathcal{H}$, such that, $m(R)= \int  A_{c}(\tau_a x) d\mu_{\infty}(x,c,a)$. Recall that $d_{x} g(a)=g\left(\tau_{a} x\right)-g(x)$.}}
\end{quote}

In the following we are going to show that $\mathcal{H}$ is not empty.

\subsection{Empirical holonomic measures}
There exist several results characterizing holonomic measures in different  contexts. For instance, in \cite{MR2765475} Proposition 3.4, in the context of Discrete Aubry-Mather theory, a complete description of the set of holonomic probabilities is given. One can found similar results for the continuous  Aubry-Mather theory. Although we can not expect the same results here because in both cases the holonomy is free
$\int_{\mathbb{T}^{d} \times \mathbb{R}^{d}} \phi(x+\tau v) \mathrm{d} \mu(x, v)=\int_{\mathbb{T}^{d} \times \mathbb{R}^{d}} \phi(x) \mathrm{d} \mu(x, v)
 $ (see \cite{MR2765475} Definition 3.1) meaning that the action $(x,v) \to x+\tau v$ span all the torus. In the IFS case the action $(x,a) \to \tau_{a} x$ is strictly bounded by the behaviour of each map $\tau_{a}$ and could be a very small fractal subset of $X$. Since we can easily find a not $\hat{\sigma}$-invariant measure whose projection is holonomic (see \cite{MR2461833} Section 4 for example), not all holonomic measure is necessarily the projection of a invariant  one for the map $(x,\bar{c},\bar{a}) \stackrel{\hat{\sigma}}{\to} (\tau_{a_0} x,\sigma \bar{c}, \sigma  \bar{a})$. For such measures the best description is via disintegration, characterizing each holonomic measure as a measure whose the marginal with respect to $x$ is the Markov measure associated to the transfer operator of a particular IFS with probabilities obtained by the conditional measures of the disintegration, see \cite{MR2461833} Section 3 and others for a construction. Therefore, we can not use the Birkhoff theorem to characterize the holonomic measures. However, we can get some similar results introducing the set of empirical holonomic measures.

We consider, for each $n \in \mathbb{N}$, $\bar{c}\in \mathcal{C}^{\mathbb{N}}$ and $\bar{a}\in \mathcal{I}^{\mathbb{N}}$ the empirical measure $\mu_{\bar{c},\bar{a},x_0}^{n}$ given by
$$\int g(x,c,a) d\mu_{\bar{c},\bar{a},x_0}^{n}= \frac{1}{n}\sum_{i=0}^{n-1} g(x_i,c_i,a_i)$$
where $x_i$ is the iteration by the IFS $R$ of some point $x_0$, controlled by $\bar{a}$.

We define
\[
\mathcal{H}_{0}
\equiv
\left\{
\mu\, \left|
\begin{array}{l}
\ \mu\ \text{is a cluster point, in the weak-* topology,}\\
\text{of the sequence}\ (\mu_{\bar{c},\bar{a},x_0}^{n})_{n\in\mathbb{N}}
\end{array}
\right.
\right\}
\]
the set of all the empirical holonomic measures.

One can easily show that the projections of $\hat{\sigma}$-invariant measures are empirical holonomic measures. Indeed, the map $\hat{\sigma}(x,\bar{c},\bar{a})=(\tau_{a_0} x,\sigma \bar{c}, \sigma  \bar{a})$ is continuous and $X\times\mathcal{C}^{\mathbb{N}}\times\mathcal{I}^{\mathbb{N}}$ is compact, thus  exist invariant measures. Moreover, if $M(\hat{\sigma})$ is the set of invariant measures for $\hat{\sigma}$, then $\rm{ex}(M(\hat{\sigma}))$, the set of ergodic measures with respect to $\hat{\sigma}$, is not empty. It is easy to see that, for any $\xi \in M(\hat{\sigma})$, the push-forward $\Pi^{\sharp}(\xi) \in Prob(X\times\mathcal{C}\times\mathcal{I})$  given by
$\int g(x,c,a) d\Pi^{\sharp}(\xi)= \int g(\Pi(x,\bar{c},\bar{a})) d\xi$, where $\Pi(x,\bar{c},\bar{a})=(x, c_0, a_0)$, is holonomic (by integrating $h(x,\bar{c},\bar{a})=w(x)$ and using the invariance of $\hat{\sigma}$).

Additionally, if $\xi \in M(\hat{\sigma})$ ergodic, then we get, from the Birkhoff theorem, that there exists a set $\Omega \subset X\times\mathcal{C}^{\mathbb{N}}\times\mathcal{I}^{\mathbb{N}}$, with $\xi(\Omega)=1$ such that, for all $(x_{0},\bar{c},\bar{a}) \in \Omega$
$$\frac{1}{n}\sum_{i=0}^{n-1} g(\Pi(\hat{\sigma}^i(x_{0},\bar{c},\bar{a}))) \to \int g(\Pi(x,\bar{c},\bar{a})) d\xi=\int g(x,c,a) d\Pi^{\sharp}(\xi)$$
for all $g\in C^0( X\times\mathcal{C}\times\mathcal{I} , \mathbb{R})$.
Since  $\frac{1}{n}\sum_{i=0}^{n-1} g(\Pi(\hat{\sigma}^i(x_{0},\bar{c},\bar{a})))=\int g(x,c,a) d\mu_{\bar{c},\bar{a},x_0}^{n}$ we conclude that $\Pi^{\sharp}(\xi) \in \mathcal{H}_{0}$, because it is a cluster point. Therefore, $\Pi^{\sharp}(\rm{ex}(M(\hat{\sigma}))) \subseteq \mathcal{H}_{0}$. As $\Pi^{\sharp}$ is a linear operator, we get $\Pi^{\sharp}(M(\hat{\sigma})) \subseteq \mathcal{H}_{0}$.\\

The next theorem shows that we can solve the IFS ergodic optimization problem in $\mathcal{H}_{0}$ but we do not know if there exist other solutions in $\mathcal{H}/\mathcal{H}_{0}$.
\begin{theorem}\label{empirical holonomic} The following properties are true:
\begin{itemize}
\item[a)] $\mathcal{H}_{0} \subset \mathcal{H}$.
\item[b)] $\int A_{c}(\tau_{a}(x)) d\mu  \leq \bar{u}$ for all $\mu \in \mathcal{H}_{0}$.
\item[c)] If $\bar{c}$ and $\bar{a}$ are optimal in $x_0$ and $\mu_{\bar{c},\bar{a},x_0}^{n} \rightharpoonup \mu_{\infty}$ ( up to subsequence $n_k \to \infty$) then $\int A_{c}(\tau_{a}(x)) d\mu_{\infty}  = \bar{u}$.
\item[d)] $\displaystyle \sup_{\mu \in \mathcal{H}_{0}} \int  A_{c}(\tau_x) d\mu(x,c,a)= \bar{u}.$
\item[e)] $\bar{u}= m(R)$.
\end{itemize}
\end{theorem}
\begin{proof}
(a) We can easily see that $\mu_{\bar{c},\bar{a},x_0}^{n}$ is a probability, but not necessarily a holonomic measure. But,  if $\mu_{\bar{c},\bar{a},x_0}^{n} \rightharpoonup \mu$,  then $\mu \in \mathcal{H}$.\\
(b) Integrating $g(x,c,a)=A_{c}(\tau_{a}(x))$ with respect to $\mu_{\bar{c},\bar{a},x_0}^{n}$,  we have
$$\int A_{c}(\tau_{a}(x)) d\mu_{\bar{c},\bar{a},x_0}^{n}= \frac{1}{n}\sum_{i=0}^{n-1} A_{c_i}(\tau_{a_i}(x_i))= \frac{1}{n}\sum_{i=0}^{n-1} A_{c_i}(x_{i+1}),$$
and using the sub-action equation we get
$$ \bar{u} \geq  d_{x_0}b(a_0) + A_{c_0}(\tau_{a_0}(x_0))=d_{x_0}b(a_0) + A_{c_0}(x_1),$$
$$ \bar{u} \geq  d_{x_1}b(a_1) + A_{c_1}(\tau_{a_1}(x_1))=d_{x_1}b(a_1) + A_{c_1}(x_2), etc.$$

Adding these sequences of inequalities we obtain
$ n \bar{u} \geq b(x_{n}) - b(x_0) + \sum_{i=0}^{n-1} A_{c_i}(x_{i+1})$,
with equality, if and only if, $\bar{c}$ and $\bar{a}$ are optimal in $x_0$.
Rewriting the above inequality we obtain
$$ \bar{u} \geq \frac{b(x_{n}) - b(x_0)}{n} + \frac{1}{n}\sum_{i=0}^{n-1} A_{c_i}(x_{i+1})= \frac{b(x_{n}) - b(x_0)}{n} + \int A_{c}(\tau_{a}(x)) d\mu_{\bar{c},\bar{a},x_0}^{n}.$$

As $b$ is continuous and $X$ is compact we can take the limit and conclude that  $\int A_{c}(\tau_{a}(x)) d\mu  \leq \bar{u}$.

(c) If $\bar{c}$ and $\bar{a}$ are optimal in $x_0$ and $\mu_{\infty} =\lim \mu_{\bar{c},\bar{a},x_0}^{n}$, then $\int A_{c}(\tau_{a}(x)) d\mu_{\infty}  = \bar{u}$.

(d) From (b) and (c) we obtain that
$\displaystyle\sup_{\mu \in \mathcal{H}_{0}} \int  A_{c}(\tau_x) d\mu(x,c,a)=\bar{u}.$

(e) From the Bellman  equation $ \bar{u} \geq  d_{x}b(a) + A_{c}(\tau_{a}(x)),$ we obtain that, for all $\mu \in \mathcal{H}$, $\int  A_{c}(\tau_x) d\mu(x,c,a) \leq \bar{u}$. From (a), $\mathcal{H}_{0} \subset \mathcal{H}$, thus $m(R)=\bar{u}$.
\end{proof}

\subsection{Empirical discounted holonomic measures}
By analogy with \cite{MR2458239} we define the set of discounted holonomic measures with discount, $0< \lambda < 1$, and trace $\nu \in Prob(X)$, as being the set
$$\mathcal{H}^{\lambda}(\nu)=\left\{ \mu \in Prob(X\times\mathcal{C}\times\mathcal{I})\; | \; \int d_{x}^{\lambda} w (a) d\mu(x,c,a)=\right.$$ $$\left.-(1-\lambda)\int w(x) d\nu(x), \; \forall w \in C^0(X,\mathbb{R}) \right\},$$
where $d_{x}^{\lambda}w(a)= \lambda w(\tau_a x) -w(x)$ is the discounted discrete differential of a continuous function $w$. By an abuse of notation, we can write $d_{x}^{1}w(a)=d_{x}w(a)$, then $\mathcal{H}^{1}(\nu)=\mathcal{H}$, for any $\nu$.
As before, we can formulate the discounted IFS ergodic optimization problem for a finite family of potentials as follows:

\begin{quote}
\emph{\textbf{Given a finite set of potentials $A_{c}: X \to \mathbb{R}$, for $c \in \mathcal{C}$, a discount $0<\lambda<1$, a trace $\nu \in Prob(X)$ and an IFS $R=(X, \tau_a)_{a \in \mathcal{I}}$, we want to find the value
$\displaystyle m_{\lambda}(R)=\sup_{\mu \in \mathcal{H}^{\lambda}(\nu)} \int  A_{c}(\tau_a x) d\mu(x,c,a),$
over the set of discounted holonomic probabilities $\mathcal{H}^{\lambda}(\nu)$ and the optimal measures $\mu_{\infty} \in \mathcal{H}^{\lambda}(\nu)$, such that, $m_{\lambda}(R)= \int  A_{c}(\tau_a x) d\mu_{\infty}(x,c,a)$.}}
\end{quote}

\begin{theorem} \label{discounted holonomic prob properties} The following properties are true:\\
a) For any trace measure $\nu \in Prob(X)$, we get $\gamma \in \mathcal{H}^{\lambda}(\nu)$, where $\gamma: C^0(X, \mathbb{R}) \to \mathbb{R}$ is given  by
$\gamma(g)=(1-\lambda) \int_{X} \sum_{i=0}^{\infty} \lambda^i g(x_i,c_i,a_i) d\nu(x),$
for a fixed pair $(\bar{c}, \bar{a}) \in\mathcal{C}^{\mathbb{N}}\times \mathcal{I}^{\mathbb{N}}$ and $x_0=x \in X$ where, $x_i$ is the  $i$-iterate by the IFS $R$ of $x_0$. In particular the set $\mathcal{H}^{\lambda}(\nu)$ is not empty.\\
b) For any trace $\nu \in Prob(X)$, we have $\displaystyle m_{\lambda}(R) =  (1-\lambda) \int v_{\lambda}(x)d\nu(x)$ where $v_{\lambda}(x)$ is the solution of the Bellman equation $\displaystyle 0=\sup_{(c,a) \in \mathcal{C}\times \mathcal{I}} A_{c}(\tau_{a}(x)) +  d_{x}^{\lambda}v_{\lambda}(a)$.\\
c) For any $z \in X$ the value of the problem for $\mathcal{H}^{\lambda}(\delta_{z})$ is $m_{\lambda}(R)=(1-\lambda) v_{\lambda}(z)$.\\
d) We have the formula $$v_{\lambda}(z)= \inf_{w \in C^0(X,\mathbb{R})}\left[ w(z) + \frac{1}{1-\lambda}\sup_{x\in X} h_{w}(x)\right],$$
where $\displaystyle h_{w}(x)= \sup_{(c,a) \in \mathcal{C}\times \mathcal{I}} \left\{d_{x}^{\lambda}w(a) + A_{c}(\tau_a x) \right\}$.\\
\end{theorem}
\begin{proof}
(a) To see that $\mathcal{H}^{\lambda}(\nu)$ is not empty, we consider the functional $\gamma: C^0(X, \mathbb{R}) \to \mathbb{R}$ given  by
$\displaystyle\gamma(g)=(1-\lambda) \int_{X} \sum_{i=0}^{\infty} \lambda^i g(x_i,c_i,a_i) d\nu(x),$
for a fixed pair $(\bar{c}, \bar{a}) \in\mathcal{C}^{\mathbb{N}}\times \mathcal{I}^{\mathbb{N}}$ and $x_0=x \in X$.
This linear functional is well defined and continuous, because $x \to \sum_{i=0}^{\infty} \lambda^i g(x_i,c_i,a_i)$ is bounded. Moreover, $\gamma$ is positive and $\gamma(1)=1$. Therefore, $\gamma$ is a probability. Computing
$\displaystyle\sum_{i=0}^{n-1} \lambda^i d_{x_i}^{\lambda}w(a_i)=\lambda^n w(x_{n}) -w(x)$
and taking the limit when $n\to\infty$ we can see that
$\displaystyle\int d_{x}^{\lambda}w(a) d\gamma(x,c,a)= -(1-\lambda) \int_{X} w(x)d\nu(x)$, thus $\gamma \in \mathcal{H}^{\lambda}(\nu)$.\\

(b) We will apply Fenchel-Rockafellarduality theorem (see Section~\ref{appendix}). From Theorem~\ref{dual ergodig with discount}, there exists at least one $\mu^{\lambda}$ satisfying
the equation $\displaystyle\int d_{x}^{\lambda}w(a) d\mu^{\lambda}(x,c,a)= -(1-\lambda) \int_{X} w(x)d\nu(x),\; w \in C^0(X,\mathbb{R})$, such that,
\begin{equation} \label{max_realize}
m_{\lambda}(R)=\int  A_{c}(\tau_a x) d\mu^{\lambda}(x,c,a).
\end{equation}
Also  from Theorem~\ref{dual ergodig with discount}, we have the formula
\begin{equation} \label{max_dual}
\inf_{w \in C^0(X,\mathbb{R})}\left[(1-\lambda)\int w(x) d\nu(x)+\sup_{x\in X,c\in\mathcal{C},a\in\mathcal{I}} \left\{d_{x}^{\lambda}w(a) + A_{c}(\tau_a x) \right\}\right] = m_{\lambda}(R).
\end{equation}
Finally, recall that the function $v_{\lambda}$ satisfy the Bellman equation
$\displaystyle 0=\sup_{(c,a) \in \mathcal{C}\times \mathcal{I}} A_{c}(\tau_{a}(x)) +  d_{x}^{\lambda}v_{\lambda}(a)$.

From the equation \eqref{max_dual} we are inspired to define the functional
\begin{equation} \label{max_func}
\psi(w)= (1-\lambda)\int w(x) d\nu(x)+\sup_{x\in X,c\in\mathcal{C},a\in\mathcal{I}} \left\{d_{x}^{\lambda}w(a) + A_{c}(\tau_a x) \right\}, \quad \text{ for } w \in C^0(X,\mathbb{R}).
\end{equation}

We notice that $\psi(v_{\lambda})= (1-\lambda)\int v_{\lambda}(x)  \;  d\nu(x)+ 0=(1-\lambda)\int v_{\lambda}(x)  \;  d\nu(x)$ and from \eqref{max_realize} we obtain
$$m_{\lambda}(R)=\int  A_{c}(\tau_a x)  \; d\mu^{\lambda}(x,c,a)= \int  A_{c}(\tau_a x) + d_{x}^{\lambda}v_{\lambda}(a) - d_{x}^{\lambda}v_{\lambda}(a) \; d\mu^{\lambda}(x,c,a)=$$
$$=\int  A_{c}(\tau_a x) + d_{x}^{\lambda}\, v_{\lambda}(a) \; d\mu^{\lambda}(x,c,a) - \int d_{x}^{\lambda}v_{\lambda}(a)  \; d\mu^{\lambda}(x,c,a) =$$
$$= \int  A_{c}(\tau_a x) + d_{x}^{\lambda} \, v_{\lambda}(a) \; d\mu^{\lambda}(x,c,a) -\left(-(1-\lambda) \int_{X}  v_{\lambda}(x)d\nu(x)\right)= $$
$$= (1-\lambda) \int_{X}  v_{\lambda}(x)d\nu(x)  +  \int  A_{c}(\tau_a x) + d_{x}^{\lambda}v_{\lambda}(a) \; d\mu^{\lambda}(x,c,a) = $$
 $$=\psi(v_{\lambda}) + \int  A_{c}(\tau_a x) + d_{x}^{\lambda}v_{\lambda}(a) \; d\mu^{\lambda}(x,c,a).$$

On the other hand, from the equation \eqref{max_dual} we know that
$\displaystyle \inf_{w \in C^0(X,\mathbb{R})}\psi(w) = m_{\lambda}(R),$
thus
$$ \psi(v_{\lambda}) + \int  A_{c}(\tau_a x) + d_{x}^{\lambda}v_{\lambda}(a) \; d\mu^{\lambda}(x,c,a) = \inf_{w \in C^0(X,\mathbb{R})}\psi(w).$$
As  $v_{\lambda} \in C^0(X,\mathbb{R})$ and $A_{c}(\tau_a x) + d_{x}^{\lambda}v_{\lambda}(a) \leq 0$ we conclude that $\int  A_{c}(\tau_a x) + d_{x}^{\lambda}v_{\lambda}(a) \; d\mu^{\lambda}(x,c,a) =0$ and $m_{\lambda}(R)= \psi(v_{\lambda})= (1-\lambda)\int v_{\lambda}(x)  \;  d\nu(x)$.\\
(c) It is a consequence of (b), when $\nu=\delta_{z}$.\\
(d) Using (b) and the Theorem~\ref{dual ergodig with discount}, we obtain the equality
$$\inf_{w \in C^0(X,\mathbb{R})}\left[(1-\lambda)\int w(x) d\nu(x)+\sup_{x\in X,c\in\mathcal{C},a\in\mathcal{I}} \left\{d_{x}^{\lambda}w(a) + A_{c}(\tau_a x) \right\}\right] =  (1-\lambda) \int_{X}  v_{\lambda}(x)d\nu(x), $$
for any trace measure $\nu$.  Using this formula for $\nu=\delta_{z} \in Prob(X)$ we can easily get the proposed representation of $v_{\lambda}$.
\end{proof}

Inspired by Theorem~\ref{discounted holonomic prob properties}, we are going to consider the particular case, $\nu=\delta_{x_0} \in Prob(X)$, where $\displaystyle x_0= {\rm argmax}_{x \in X} v_{\lambda}(x)$.
\begin{definition}\label{empirical discounted prob measure delta_0}
We consider, for each $\bar{c} \in \mathcal{C}^{\mathbb{N}}$ and $\bar{a} \in \mathcal{I}^{\mathbb{N}}$, the \textbf{empirical discounted probability measure} $\mu_{x_0,\bar{c},\bar{a}}^{\lambda}$ given by
$\displaystyle \int g(x,c,a) d\mu_{x_0,\bar{c},\bar{a}}^{\lambda}= (1-\lambda) \sum_{i=0}^{\infty} \lambda^i g(x_i,c_i,a_i)$
where, $x_i$ is the  $i$-iterate by the IFS $R$ of some point $x_0 \in X$.
\end{definition}
These measures are quite different from the  empirical  ones because each one is not an average in $n$, or subsequences, but they are defined only for $\lambda <1$. Moreover, if $ \mu=\mu_{x_0,\bar{c},\bar{a}}^{\lambda}$, then
$$\int  A_{c}(\tau_a x) d\mu(x,c,a)=   (1-\lambda)\sum_{i=0}^{\infty} \lambda^i A_{c_i}(\tau_{a_i} x_i) = (1-\lambda) S_{x_0}(\bar{c},\bar{a}),$$ showing the connection between this maximization problem and the superior boundary of the set $\Lambda$.

\begin{theorem}\label{solve dicounted x_0} We consider, for each  $\bar{c} \in \mathcal{C}^{\mathbb{N}}$ and $\bar{a} \in \mathcal{I}^{\mathbb{N}}$, the empirical discounted probability measure $\mu_{x_0,\bar{c},\bar{a}}^{\lambda}$. Then, the following statements are true:\\\\
 a) $\mu_{x_0,\bar{c},\bar{a}}^{\lambda} \in \mathcal{H}^{\lambda}(\delta_{x_0})$. \\
 b) Consider $\displaystyle x_0= {\rm argmax}_{x \in X} v_{\lambda}(x)$ and $\nu=\delta_{x_0} \in Prob(X)$. If the pair $(\bar{c}, \bar{a}) \in \mathcal{C}^{\mathbb{N}}\times\mathcal{I}^{\mathbb{N}}$ is optimal then, $m_{\lambda}(R)=(1-\lambda) \max_{x \in X}v_{\lambda}(x)$ and  $\mu^{\lambda}=\mu_{x_0,\bar{c},\bar{a}}^{\lambda} \in \mathcal{H}^{\lambda}(\delta_{x_0})$ solves the discounted IFS ergodic optimization problem for the finite family of potentials.\\
\end{theorem}
\begin{proof}
(a) Note that
$$\int d_{x}^{\lambda}w(a) d\mu_{x_0,\bar{c},\bar{a}}^{\lambda}= (1-\lambda)\sum_{i=0}^{+\infty} \lambda^i d_{x_i}^{\lambda}w(a_i)=(1-\lambda)\lim_{n\to+\infty} (\lambda^n w(x_{n}) -w(x_0))=$$ $$=-(1-\lambda) w(x_0)= -(1-\lambda)\int w(x)d\delta_{x_0}(x).$$
Thus, $\mu_{x_0,\bar{c},\bar{a}}^{\lambda} \in\mathcal{H}^{\lambda}(\delta_{x_0})$.

(b) We recall that the integral $\displaystyle \int A_{c}(\tau_a x) d\mu_{x_0,\bar{c},\bar{a}}^{\lambda}$ is the limit
$$\int A_{c}(\tau_a x) d\mu_{x_0,\bar{c},\bar{a}}^{\lambda}= (1-\lambda)\lim_{n\to+\infty} \sum_{i=0}^{n-1} \lambda^i A_{c_i}(\tau_{a_i} x_i).$$

We can use the Bellman equation
$ \displaystyle v_{\lambda}(x)=\sup_{(c,a) \in \mathcal{C}\times \mathcal{I}} A_{c}(\tau_{a}(x)) +  \lambda v_{\lambda}(\tau_{a}(x))$
to construct a special discounted measure. We observe that the Bellman equation
could be rewritten as
$$ 0=\sup_{(c,a) \in \mathcal{C}\times \mathcal{I}} A_{c}(\tau_{a}(x)) +  \lambda v_{\lambda}(\tau_{a}(x))-v_{\lambda}(x)=\sup_{(c,a) \in \mathcal{C}\times \mathcal{I}} A_{c}(\tau_{a}(x)) +  d_{x}^{\lambda}v_{\lambda}(a).$$
In particular, $A_{c}(\tau_{a}(x))  \leq  -d_{x}^{\lambda}v_{\lambda}(a), \; \forall (c, a) \in \mathcal{C}\times \mathcal{I}$, and the equality is attained for some pair $(c_0, a_0) \in \mathcal{C}\times \mathcal{I}$ in each point $x \in X$.

From the above inequality we obtain that for any measure $\mu$ the following inequality holds
$$\int A_{c}(\tau_{a}(x)) d\mu(x,c,a) \leq \int -d_{x}^{\lambda}v_{\lambda}(a)d\mu(x,c,a).$$
If we additionally suppose that  $\mu \in \mathcal{H}^{\lambda}(\nu)$, then
$\displaystyle \int A_{c}(\tau_{a}(x)) d\mu(x,c,a) \leq -(1-\lambda) \int -v_{\lambda}(x)d\nu(x).$
Thus, $m_{\lambda}(R) \leq (1-\lambda) \int v_{\lambda}(x)d\nu(x)=(1-\lambda)   v_{\lambda}(x_0)$
in the case where $\nu =\delta_{x_0}$.

We take the measure $\mu^{\lambda}=\mu_{x_0,\bar{c},\bar{a}}^{\lambda}$ by choosing $\displaystyle x_0= {\rm argmax}_{x \in X} v_{\lambda}(x)$ and $\bar{c}$ and $\bar{a}$ in, such way, that $ 0= A_{c_i}(\tau_{a_i}(x_i)) +  d_{x_i}^{\lambda}v_{\lambda}(a_i)$, or $ A_{c_i}(\tau_{a_i}(x_i)) =-  d_{x_i}^{\lambda}v_{\lambda}(a_i)$, for all $i\in \mathbb{N}$. Thus,
$$\int A_{c}(\tau_a x) d\mu^{\lambda}=(1-\lambda)\lim_{n \to \infty}\sum_{i=0}^{n-1} \lambda^i (-  d_{x_i}^{\lambda}v_{\lambda}(a_i))=(1-\lambda) \max_{x \in X}v_{\lambda}(x).$$
It means that $m_{\lambda}(R)=(1-\lambda) \max_{x \in X}v_{\lambda}(x)$ and then $\mu^{\lambda}$ solves the discounted ergodic optimization problem for the finite family of potentials for $\nu=\delta_{x_0} \in Prob(X)$.
\end{proof}

\begin{remark}
  From  \cite{Cioletti_2019} we know that $\displaystyle (1-\lambda) \max_{x \in X}v_{\lambda}(x) \to \bar{u}$, thus if $\mu^{\lambda} \to \mu$, through some subsequence $\lambda_j \to 1$, then $\mu \in \mathcal{H}$ and $\displaystyle\int A_{c}(\tau_a x) d\mu=\lim_{\lambda_j \to 1}\int A_{c}(\tau_a x) d\mu^{\lambda_j}= (1-\lambda_j) \max_{x \in X}v_{\lambda_j}(x)= \bar{u}$. So, we can obtain maximizing holonomic measures as limit of discounted holonomic measures with trace $\delta_{x_0}$, arising from the solutions of Bellman's equation.
\end{remark}
\subsection{On the support of the maximizing measures and invariance properties}

We recall that the continuous function $v_{\lambda}(x)$ satisfies the Bellman equation
$ \displaystyle v_{\lambda}(x)=\sup_{(c,a) \in \mathcal{C}\times \mathcal{I}} A_{c}(\tau_{a}(x)) +  \lambda v_{\lambda}(\tau_{a}(x))$, and also satisfies the equation
$\displaystyle 0=\sup_{(c,a) \in \mathcal{C}\times \mathcal{I}} A_{c}(\tau_{a}(x)) +  d_{x}^{\lambda}v_{\lambda}(a).$
In particular, $A_{c}(\tau_{a}(x))  \leq  -d_{x}^{\lambda}v_{\lambda}(a), \; \forall (c, a) \in \mathcal{C}\times \mathcal{I}$.

From this property we will get a characterization of the support of discounted holonomic maximizing measures and its invariance.
\begin{theorem}\label{support of discounted holonomic maximizing measures} Let $v_{\lambda}$ be the only solution of the Bellman equation. Then, a measure $\mu^{\lambda} \in \mathcal{H}^{\lambda}(\nu)$ is a maximizing measure, if, and only if, $\supp{\mu^{\lambda}} \subseteq \{(x,c,a) \; | \; 0= A_{c}(\tau_{a}(x)) +  d_{x}^{\lambda}v_{\lambda}(a)\}$. In particular, the superior boundary of $\Lambda$ is essentially backwards invariant, that is,  $G_{c}(\tau_{a}(x), v_{\lambda}(\tau_{a}(x)))=(x, v_{\lambda}(x))$ for $\mu^{\lambda}$.a.e. $(x,c,a)$.
\end{theorem}
\begin{proof}
The proof is an adaptation of the argument presented in \cite{MR1841880} and \cite{MR1855838}. We also use the Theorem~\ref{discounted holonomic prob properties} (b) to claim that, for any trace $\nu \in Prob(X)$,  we have $ m_{\lambda}(R) =  (1-\lambda) \int v_{\lambda}(x)d\nu(x)$.

Consider the continuous function $b_{\lambda}(x,c,a)=A_{c}(\tau_{a}(x)) +  d_{x}^{\lambda}v_{\lambda}(a) \leq 0$. If $\supp{\mu^{\lambda}} \subseteq \{(x,c,a) \; | \; 0= A_{c}(\tau_{a}(x)) +  d_{x}^{\lambda}v_{\lambda}(a)\}$, integrating this function with respect to $\mu^{\lambda} \in \mathcal{H}^{\lambda}(\nu)$ we get that $\mu_{\lambda}$ is maximizing. Reciprocally, if $\mu_{\lambda}$ is maximizing, then
$\displaystyle \int A_{c}\left(\tau_{a}(x)\right) d \mu^{\lambda}(x, c, a)=m_{\lambda}(R)=$
$\displaystyle (1-\lambda) \int v_{\lambda}(x) d \nu(x)=-\int d_{x}^{\lambda} v_{\lambda}(a) d \mu^{\lambda}(x, c, a)$ implies $$\int b_{\lambda}(x, c, a) d \mu^{\lambda}(x, c, a)=\int A_{c}\left(\tau_{a}(x)\right)+d_{x}^{\lambda} v_{\lambda}(a) d \mu^{\lambda}(x, c, a)=0$$
therefore $\supp{\mu^{\lambda}} \subseteq \{(x,c,a) \; | \; 0= A_{c}(\tau_{a}(x)) +  d_{x}^{\lambda}v_{\lambda}(a)\}$ because $b_{\lambda}(x,c,a) \leq 0$.
\end{proof}

When we take the discounting limit $\lambda \to 1$   we get $m_{\lambda}(R) \to m(R)=m$ and  the equation $ \displaystyle v(x)=\sup_{(c,a) \in \mathcal{C}\times \mathcal{I}} (A_{c}(\tau_{a}(x)) - m) +   v(\tau_{a}(x))$ holds
for any cluster point $v_{\lambda} \to v$. This could be rewritten as
$\displaystyle 0=\sup_{(c,a) \in \mathcal{C}\times \mathcal{I}} (A_{c}(\tau_{a}(x))  - m) + d_{x}v(a).$
In particular, $(A_{c}(\tau_{a}(x))  - m) \leq  -d_{x}v(a), \; \forall (c, a) \in \mathcal{C}\times \mathcal{I}$.

From this property we get a characterization of the support of holonomic maximizing measures and its invariance under the modified maps $\bar{G}_{c}(x,y)=(T(x), (A_{c}(x) - m) + y)$.
\begin{theorem}\label{support of holonomic maximizing measures} Let $v$ be any continuous solution of the equation $\displaystyle 0=\sup_{(c,a) \in \mathcal{C}\times \mathcal{I}} (A_{c}(\tau_{a}(x))  - m) + d_{x}v(a)$. Then, a measure $\mu \in \mathcal{H}$ is a maximizing measure, if, and only if, $\supp{\mu} \subseteq \{(x,c,a) \; | \; 0= (A_{c}(\tau_{a}(x))-m) +  d_{x}v(a)\}$. In particular, the graph $\{ (x, v(x)) \; | \; x \in \mathbf{S}^1\}$ is essentially backwards invariant, that is,  $\bar{G}_{c}(\tau_{a}(x), v(\tau_{a}(x)))=(x, v(x))$, for $\mu$.a.e. $(x,c,a)$.
\end{theorem}
\begin{proof}
The proof is analogous to the proof of Theorem~\ref{support of discounted holonomic maximizing measures} by considering the continuous function $b(x,c,a)=(A_{c}(\tau_{a}(x))-m) +  d_{x}v(a) \leq 0$. We also use the fact that $\displaystyle \sup_{\mu \in \mathcal{H}} \int  A_{c}(\tau_x a) d\mu(x,c,a)=  m$.
\end{proof}

\section{Appendix: Duality for Discounted Holonomic Measures}\label{appendix}
For the sake of completeness, we state the duality result used in the Theorem~\ref{discounted holonomic prob properties}.  See \cite{MR1964483} for a proof of the  Fenchel-Rockafellar duality theorem. See \cite{MR1901094} and \cite{MR2458239} for applications in optimization and variational problems.

\begin{theorem}[\textbf{Fenchel-Rockafellar duality}]\label{FR}
Suppose that $E$ is  a normed vector space,  $\Gamma$ and $\Phi$ are two convex functions defined on $E$ taking values in $\mathbb{R}\cup \{+\infty\}$. Denote $\Gamma^{\ast}$ and  $\Phi^{\ast}$, respectively, the Legendre-Fenchel transforms of  $\Gamma$ and $\Phi$.
Suppose there exists  $\psi_0\in E$, such that, $\Gamma(\psi_0)<+\infty,\, \Phi(\psi_0)<+\infty$, and that $\Gamma$ is continuous on $\psi_0$.
Then,
\begin{equation}
\inf_{\psi \in E}[\Gamma(\psi)+\Phi(\psi)]=\sup_{\pi\in E^{*}}[-\Gamma^{*}(-\pi)-\Phi^{*}(\pi)] \label{rockafeller}
\end{equation}
Moreover, the supremum in ($\ref{rockafeller}$) is attained in at least one element in $\pi \in E^*$.
\end{theorem}
Given a finite set of potentials $A_{c}: X \to \mathbb{R}$ for $c \in \mathcal{C}$, a discount $0<\lambda<1$, a trace $\nu \in Prob(X)$ and an IFS $R=(X, \tau_a)_{a \in \mathcal{I}}$ we want to find the dual of
$\displaystyle m_{\lambda}(R)=\sup_{\mu \in \mathcal{H}^{\lambda}(\nu)} \int  A_{c}(\tau_a x) d\mu(x,c,a).$
Recall that the set of discounted holonomic probabilities $\mathcal{H}^{\lambda}(\nu) $ is formed by the measures $\mu$ such that

\begin{equation}\label{holonomic}
\int d_{x}^{\lambda} w (a) d\mu(x,c,a) =-(1-\lambda)\int w(x) d\nu(x),
\end{equation}
for all $w \in C^0(X,\mathbb{R})$. Here, $d_{x}^{\lambda}w(a)= \lambda w(\tau_a x) -w(x)$ is the discounted discrete differential of a continuous function $w$.

\begin{theorem} \label{dual ergodig with discount} Under the above hypothesis, there exists at least one $\mu$ satisfying the condition \eqref{holonomic}, such that,
$m_{\lambda}(R)=\max_{\mu \in \mathcal{H}^{\lambda}(\nu)} \int  A_{c}(\tau_a x) d\mu(x,c,a)$. Besides,
\begin{equation}\label{dualgeral}
\inf_{w \in C^0(X,\mathbb{R})}\left[(1-\lambda)\int w(x) d\nu(x)+\sup_{x\in X,c\in\mathcal{C},a\in\mathcal{I}} \left\{d_{x}^{\lambda}w(a) + A_{c}(\tau_a x) \right\}\right] = \max_{\mu \in \mathcal{H}^{\lambda}(\nu)} \int  A_{c}(\tau_a x) d\mu(x,c,a).
\end{equation}
\end{theorem}
\begin{proof}
We define for $\psi \in C(X\times \mathcal{C}\times \mathcal{I})$ the maps
$\displaystyle  \Gamma(\psi)= \sup_{X\times \mathcal{C}\times \mathcal{I}}\left\{ \psi(x,c,a) + A_{c}(\tau_a x) \right\},$
and
$$\Phi(\psi) = \left\{\begin{array}{ll} (1-\lambda)\int w(x) d\nu(x), & \text{ if } \psi(x,c,a) = d_{x}^{\lambda}w(a) \text{ for some } w(x)  \\ +\infty, & \text{ otherwise } \end{array}\right..$$

$\Phi$ is well defined because if $\pi_0$ is any measure with trace $\nu$  and if $\psi = d_{x}^{\lambda}w_1(a)= d_{x}^{\lambda}w_2(a)$, then, $\int d_{x}^{\lambda}w_1(a) d\pi_0= \int d_{x}^{\lambda}w_2(a) d\pi_0$, therefore
$\displaystyle (1-\lambda)\int w_1(x) d\nu(x)=(1-\lambda)\int w_2(x) d\nu(x).$

The functions $\Gamma$ and $\Phi$ are trivially convex. To fulfill the hypothesis of Theorem \ref{FR} we need to find one function $\psi$, such that,  $\Gamma(\psi),\Phi(\psi)<\infty$ and $\Gamma$ is continuous in $\psi$. Then, take any $\psi$ in the form $\psi=d_{x}^{\lambda}w(a)$, for some $w(x)$, and use that $\Gamma$ is continuous for the supremum norm.

Now we study $\Gamma^{*}(\pi)$.
First suppose that $\pi$ is not a positive functional. Then, there exists $u\leq 0$, such that, $\pi(u)>0$. We write $\psi_t(x,c,a) = t u(x,c,a) - A_{c}(\tau_a x)$. Then,
\[\Gamma^{*}(\pi)= \sup_{\psi}\left\{ \int \psi d\pi  - \Gamma(\psi) \right\}
\geq \limsup_{t\to\infty} \int \psi_t(x,y) d\pi  - \Gamma(\psi_t)\]
\[=  \limsup_{t\to\infty}\int t\, u - A_{c}(\tau_a x) d\pi  - \Gamma(t\, u -A_{c}(\tau_a x))
=\limsup_{t\to\infty} \int- A_{c}(\tau_a x) d\pi+t\int  u d\pi  - t \sup{u}\]
\[\geq \limsup_{t\to\infty} \int -A_{c}(\tau_a x) d\pi+ t\int  u d\pi =+\infty. \]

Suppose that $\pi$ is a positive functional. If $\int 1d\pi\neq 1$, then for any $t \in \mathbb{R}$:
\[\Gamma^{*}(\pi)= \sup_{\psi}\left\{ \int \psi d\pi  - \Gamma(\psi) \right\}
\geq \left\{ \int t -A_{c}(\tau_a x) d\pi  - \Gamma(t-A_{c}(\tau_a x)) \right\}\]
\[= \int -A_{c}(\tau_a x)\, d\pi + \int t d\pi  - t = \int -A_{c}(\tau_a x)\, d\pi + t\left(\int 1 d\pi  - 1\right)  . \]
If $\int 1\,d\pi>1$, we get $\Gamma^{*}(\pi) \geq \limsup_{t\to +\infty}\int -A_{c}(\tau_a x)\, d\pi + t(\int 1 d\pi  - 1) = +\infty$.
If $\int 1\,d\pi<1$, we get $\Gamma^{*}(\pi) \geq \limsup_{t\to -\infty}\int- A_{c}(\tau_a x)\, d\pi + t(\int 1 d\pi  - 1) = +\infty$. We conclude that if $\int 1 d\pi \neq 1 $, then, $\Gamma^{*}(\pi)=+\infty$.

Now, we suppose that $\pi$ is a probability. Then:
$$\Gamma^{*}(\pi)= \sup_{\psi}\left\{ \int \psi d\pi  - \Gamma(\psi) \right\} = $$
$$ =\sup_{\psi}\left\{ \int -A_{c}(\tau_a x) d\pi  +\int \psi(x,y)+ A_{c}(\tau_a x) d\pi  - \Gamma(\psi) \right\} = $$
$$ =\sup_{\psi}\left\{ \int -A_{c}(\tau_a x) d\pi + \int (\psi + A_{c}(\tau_a x)) - \Gamma(\psi) d\pi  \right\} \leq $$
$$ \leq\sup_{\psi}\left\{ \int -A_{c}(\tau_a x) d\pi + 0 \right\}  = \int -A_{c}(\tau_a x) d\pi.$$
On the other hand,
$$\Gamma^{*}(\pi)= \sup_{\psi}\left\{ \int \psi  d\pi  - \Gamma(\psi) \right\}  \geq  \left\{ \int -A_{c}(\tau_a x) d\pi  - \Gamma(-A_{c}(\tau_a x)) \right\}= \int -A_{c}(\tau_a x) d\pi$$
Thus, we conclude that $$\Gamma^{*}(\pi)=\left\{\begin{array}{ll} -\int A_{c}(\tau_a x) d\pi(x,c,a) & \text{if}\, \pi\, \text{is a probability}\\ +\infty & \text{otherwise}\end{array}\right. .$$

We recall that $\Phi(\psi) =  (1-\lambda)\int w(x) d\nu(x)$ if $ \psi(x,c,a) = d_{x}^{\lambda}w(a)$, for some $w(x)$ and $ +\infty$, otherwise.
As the Legendre transform is
$\displaystyle \Phi^{*}(-\pi) =  \sup_{\psi}\left[\langle -\pi, \psi \rangle - \Phi(\psi)\right],$
we must consider the supremum just for $\psi(x,c,a) = d_{x}^{\lambda}w(a) $, thus,
\[ \sup_{w}\left[ -\int d_{x}^{\lambda}w(a)\, d\pi - ((1-\lambda)\int w(x) d\nu(x))\right]=0,\] so
\[\Phi^{*}(-\pi)=  \left\{\begin{array}{ll} 0,\,& \text{ if  } \pi \text{ is discounted holonomic }\\ +\infty, & \text{ otherwise} \end{array}\right..\]
Then,
\[ \inf_{\psi \in E}[\Gamma(\psi)+\Phi(\psi)]=
\inf_{w(x)}\left[(1-\lambda)\int w(x) d\nu(x) + \sup_{(x,c,a)} \left\{d_{x}^{\lambda}w(a) + A_{c}(\tau_a x) \right\}\right] \]
and
\[\sup_{\pi\in E^{*}}[-\Gamma^{*}(-\pi)-\Phi^{*}(\pi)] \stackrel{\pi  \Leftrightarrow -\pi}{=} \sup_{\pi\in E^{*}}[ -\Gamma^{*}(\pi)-\Phi^{*}(-\pi)]   = \sup_{\pi \in \mathcal{H}^{\lambda}(\nu)}  \int   A_{c}(\tau_a x)   d\pi.\]
From equation~\eqref{rockafeller} we conclude the proof.
\end{proof}

\textbf{Acknowledgments}: \emph{I would like to thank Professor Leandro Cioletti for having read the manuscript and contributed countless suggestions for its improvement. Thanks also to the anonymous referee, whose suggestions greatly improved the original version of the article.}



\end{document}